\documentclass[a4paper,12pt]{amsart}
\usepackage{amsmath,amsthm,amssymb}
\usepackage{graphicx,ascmac}
\usepackage[all]{xy}
\usepackage{txfonts} 
\input xypic 

\theoremstyle{definition}
\newtheorem{Thm}{Theorem}[section]

\newtheorem{Rem}[Thm]{Remark}
\newtheorem{Lem}[Thm]{Lemma}
\newtheorem{Cor}[Thm]{Corollary}
\newtheorem{Prop}[Thm]{Proposition}

\newcommand{\Z}{\mathbb{Z}}

\newcommand{\F}{\mathbb{F}}
\newcommand{\Sp}{S\!p}
\newcommand{\Spin}{S\!pin}
\newcommand{\Ss}{S\!s}
\newcommand{\SO}{S\!O}
\newcommand{\SU}{S\!U}
\newcommand{\fib}{\mathrm{fib}}
 
\newcommand{\qqed}{\hfill\Box} 

\newcommand{\namedright}[3]{\ensuremath{#1\stackrel{#2}
 {\longrightarrow}#3}}
\newcommand{\nameddright}[5]{\ensuremath{#1\stackrel{#2}
 {\longrightarrow}#3\stackrel{#4}{\longrightarrow}#5}}
\newcommand{\namedddright}[7]{\ensuremath{#1\stackrel{#2}
 {\longrightarrow}#3\stackrel{#4}{\longrightarrow}#5
  \stackrel{#6}{\longrightarrow}#7}}

\newcommand{\larrow}{\relbar\!\!\relbar\!\!\rightarrow}
\newcommand{\llarrow}{\relbar\!\!\relbar\!\!\larrow}
\newcommand{\lllarrow}{\relbar\!\!\relbar\!\!\llarrow}

\newcommand{\lnamedright}[3]{\ensuremath{#1\stackrel{#2}
 {\larrow}#3}}

\newcommand{\llnamedright}[3]{\ensuremath{#1\stackrel{#2}
 {\llarrow}#3}}
\newcommand{\llnameddright}[5]{\ensuremath{#1\stackrel{#2}
 {\llarrow}#3\stackrel{#4}{\llarrow}#5}}
\newcommand{\llnamedddright}[7]{\ensuremath{#1\stackrel{#2}
 {\llarrow}#3\stackrel{#4}{\llarrow}#5
  \stackrel{#6}{\llarrow}#7}}
\newcommand{\lllnamedright}[3]{\ensuremath{#1\stackrel{#2}
 {\lllarrow}#3}}

\title{Mod $p$ decompositions of the loop spaces of compact symmetric spaces}
\date{\today}
\author{Shizuo KAJI}
\thanks{The first named author was partially supported by KAKENHI, Grant-in-Aid for Young 
     Scientists (B) 26800043.}
\address{Department of Mathematical Sciences,
Faculty of Science \endgraf
Yamaguchi University  \endgraf 
1677-1, Yoshida, Yamaguchi 753-8512, Japan}
\email{skaji@yamaguchi-u.ac.jp}
\author{Akihiro Ohsita}
\address{Faculty of Economics,
Osaka University of Economics \endgraf
2-2-8 Osumi, Hiogashiyodogawa Ward, Osaka 533-8533,
Japan}
\email{ohsita@osaka-ue.ac.jp}
\author{Stephen Theriault}
\address{School of Mathematics, University of Southampton, 
    Southampton SO17 1BJ, United Kingdom}
\email{S.D.Theriault@soton.ac.uk}

\begin{document}

\maketitle

\begin{abstract}
We give $p$-local homotopy decompositions of the loop spaces of compact, 
simply-connected symmetric spaces for quasi-regular primes.
The factors are spheres, sphere bundles over spheres, and their loop spaces. 
As an application, upper bounds for the homotopy exponents are determined. 
\end{abstract}

\section{Introduction} 

If $X$ is a topological space and there is a homotopy equivalence 
$X\simeq A\times B$ then there are induced isomorphisms of 
homotopy groups $\pi_{m}(X)\cong\pi_{m}(A)\oplus\pi_{m}(B)$ for 
every $m\geq 1$. So in order to determine the homotopy groups 
of a space it is useful to first try to decompose it as a product, 
up to homotopy equivalence. Ideally, the factors are simpler spaces 
which are easy to recognize, so that one can deduce homotopy group 
information about the original space $X$ from known information 
about the factors. This approach has been very successful in obtaining 
important information about the homotopy groups of Lie groups~\cite{Harris3,MNT}, 
Moore spaces~\cite{CMN1}, finite $H$-spaces~\cite{CN}, and 
certain manifolds~\cite{Beben-Wu,Beben-Theriault}. 

In practise, it helps if the initial space $X$ is an $H$-space. Then the 
continuous multiplication can be used to multiply together maps 
from potential factors. For this reason, it is often the loop space 
of the original space that is decomposed up to homotopy, as looping 
introduces a multiplication and it simply shifts the homotopy groups 
of $X$ down one dimension. It also helps to localize 
at a prime $p$, or rationally, in order to simplify the calculations while 
retaining $p$-primary features of the homotopy groups. 

From now on, let $p$ be an odd prime and assume that all spaces and maps 
have been localized at~$p$. Harris~\cite{Harris3} and Mimura, Nishida and 
Toda~\cite{MNT} gave $p$-local homotopy decompositions of torsion free 
simply-connected, simple compact Lie groups into products of irreducible 
factors. These were used, for example, in~\cite{MNT} to calculate the 
$p$-primary homotopy groups of the Lie group through a range, 
in~\cite{Bendersky-Davis-Mimura} to calculate the $v_{1}$-periodic 
homotopy groups in certain cases, and  in~\cite{Davis-Theriault} to determine 
bounds on the homotopy exponents in certain cases. 
Here, the $p$-primary homotopy exponent of a space $X$ is the least power 
of $p$ that annihilates the $p$-torsion in $\pi_{\ast}(X)$. 

It is natural to extend the decomposition approach to other spaces related 
to Lie groups. Some work has been done to determine homotopy decompositions 
of the loops on certain homogeneous spaces~\cite{Beben,Grbic-Zhao} and analyze the 
exponent implications. In this paper we consider the loops on symmetric 
spaces with an eye towards deducing exponent information. 

Compact, irreducible, simply-connected Riemannian symmetric spaces were 
classified by Cartan~\cite{Cartan1,Cartan2} and an explicit list as homogeneous spaces 
was given in~\cite{Ishitoya-Toda}. In an ad-hoc manner, the homotopy groups 
of symmetric spaces have been studied in several papers, for 
example \cite{Beben,Burns,Harris1,Harris2,Hirato-Kachi-Mimura,Mimura,Oshima,Terzic}. 
We give a more systematic approach. 

A compact Lie group is \emph{quasi-$p$-regular} 
if it is $p$-locally homotopy equivalent to a product of spheres and 
sphere bundles over spheres. Let $G/H$ be a compact, irreducible, 
simply-connected Riemannian symmetric space where $G$ 
is quasi-$p$-regular. Then for $p\geq 5$ we obtain $p$-local homotopy decompositions 
for $\Omega(G/H)$, which are stated explicitly in Theorems~\ref{thm:classical}  
and~\ref{thm:exceptional}. It is notable that in all the decompositions, the factors 
are spheres, sphere bundles over spheres, and the loops on these spaces. 

The key to our method is to replace the fibration 
\begin{equation} 
  \label{GHfib} 
  \nameddright{\Omega(G/H)}{}{H}{}{G} 
\end{equation}  
with a homotopy equivalent one 
\begin{equation} 
  \label{CNfib} 
  \llnameddright{\prod(\fib(M(q_{i})))}{}{\prod M(A'_{i})}{\prod M(q_{i})} 
       {\prod M(A_{i})} 
\end{equation}  
using Cohen and Neisendorfer's construction of finite $H$-spaces~\cite{CN}
(see Theorem \ref{CN}). 
Here, (\ref{CNfib})~is an $H$-fibration with a different $H$-structure from 
that in~(\ref{GHfib}), but the maps $M(q_{i})$ are simple enough to allow us 
to identify their homotopy fibres. 

The paper is organized as follows. In Sections~$2$ through $4$ we 
obtain the homotopy fibration~(\ref{CNfib}) from~(\ref{GHfib}), and 
prove properties about it. In Section~$5$ we identify the maps $q_{i}$ 
in a case-by-case analysis, and thereby obtain a homotopy decomposition 
for $\Omega(G/H)$. In Section~$6$ we test the boundaries of our methods: 
examples are given to show that our methods can sometimes be extended 
to non-quasi-$p$-regular cases  and sometimes not; we also give an 
example to show that our loop space decompositions sometimes 
cannot be delooped. In Section~$7$ we use the homotopy decompositions 
of $\Omega(G/H)$ to deduce homotopy exponent bounds for $G/H$. 

The authors would like to thank the referee for suggesting improvements 
and for pointing out a mistake in an early version of the paper.

\section{A decomposition method} 

Let $G$ and $H$ be Lie groups and let 
\(\varphi\colon\namedright{H}{}{G}\) 
be a group homomorphism. 
In this section we describe a method 
for producing a homotopy decomposition of the homotopy fibre of
$\varphi$ when 
both $G$ and $H$ are quasi-$p$-regular. 
In the case when $\varphi$ is a group inclusion,
this gives a homotopy decomposition of $\Omega(G/H)$.
To do so, we first need some 
preliminary information. 

The following is a consequence of the James construction~\cite{J}. 
For a path-connected, pointed space~$X$, let 
\(E\colon\namedright{X}{}{\Omega\Sigma X}\) 
be the suspension map, which is adjoint to the identity map on $\Sigma X$. 

\begin{Thm} 
   \label{James} 
   Let $X$ be a path-connected space. Let $Y$ be a homotopy associative 
   $H$-space and suppose that there is a map 
   \(f\colon\namedright{X}{}{Y}\). 
   Then there is an extension 
   \[\diagram 
          X\rto^-{f}\dto^{E} & Y \\ 
         \Omega\Sigma X\urto_-{\overline{f}} & 
     \enddiagram\] 
   where $\overline{f}$ is an $H$-map and it is the unique $H$-map 
   (up to homotopy) with the property that $\overline{f}\circ E\simeq f$.~$\qqed$ 
\end{Thm} 
 
Next, Cohen and Neisendorfer~\cite{CN} gave a construction of finite 
$p$-local $H$-spaces satisfying many useful properties. The ones we 
need are listed below. For a $\mathbb{Z}/p\mathbb{Z}$-vector space $V$, 
let $\Lambda(V)$ be the exterior algebra on $V$. Take homology with 
mod-$p$ coefficients. 

\begin{Thm} 
   \label{CN} 
   Fix a prime $p$. Let $\mathcal{C}_{p}$ be the collection of $CW$-complexes 
   consisting of $\ell$ odd dimensional cells, where $\ell<p-1$. 
   If $A\in\mathcal{C}_{p}$ then there is a finite $H$-space $M(A)$ with the 
   following properties: 
   \begin{itemize} 
      \item[(a)] there is an isomorphism of Hopf algebras 
                     $H_{\ast}(M(A))\cong\Lambda(\widetilde{H}_{\ast}(A))$; 
      \item[(b)] there are maps 
                      \(\nameddright{M(A)}{s}{\Omega\Sigma A}{\rho}{M(A)}\) 
                     such that $\rho\circ s$ is homotopic to the identity map on $M(A)$; 
      \item [(c)] the composite 
                      \(\nameddright{A}{E}{\Omega\Sigma A}{\rho}{M(A)}\) 
                      induces the inclusion of the generating set in homology. 
   \end{itemize} 
   Further, if $A,A',A''\in\mathcal{C}_{p}$ then: 
   \begin{itemize} 
      \item[(d)] a map 
                      \(f\colon\namedright{A'}{}{A}\) 
                      induces a map 
                      \(M(f)\colon\namedright{M(A')}{}{M(A)}\); 
      \item[(e)] the maps $\rho$ and $s$ in part~(b) are natural for maps 
                      \(f\colon\namedright{A'}{}{A}\); 
      \item[(f)] if there is a homotopy cofibration 
                      \(\nameddright{A'}{}{A}{}{A''}\) 
                      then there is a homotopy fibration 
                      \(\nameddright{M(A')}{}{M(A)}{}{M(A'')}\). 
   \end{itemize} 
   $\qqed$ 
\end{Thm} 

It will help to have some information about $s_{\ast}$. Let $a$ be 
the composite 
\[a\colon\nameddright{A}{E}{\Omega\Sigma A}{\rho}{M(A)}\] 
and let $\overline{E}$ be the composite 
\[\overline{E}\colon\nameddright{A}{a}{M(A)}{s}{\Omega\Sigma A}.\] 
It may not be the case that $\overline{E}$ is homotopic to $E$. 
However, we will show that they induce the same map in homology 
modulo commutators. Recall by the Bott Samelson Theorem that 
$H_{\ast}(\Omega\Sigma A)\cong T(\widetilde{H}_{\ast}(A))$, where 
$T(\ \ )$ is the free tensor algebra functor. It is well known that 
for a $\mathbb{Z}/p\mathbb{Z}$-vector space $V$ there is an 
algebra isomorphism $T(V)\cong UL\langle V\rangle$ where 
$L\langle V\rangle$ is the free Lie algebra generated by $V$ and 
$U$ is the universal enveloping algebra functor. Thus there is an 
algebra isomorphism 
$H_{\ast}(\Omega\Sigma A)\cong UL\langle\widetilde{H}_{\ast}(X)\rangle$. 

\begin{Lem} 
   \label{modcomm} 
   We have $(\overline{E})_{\ast}=E_{\ast}$ modulo commutators in 
   $UL\langle\widetilde{H}_{\ast}(X)\rangle$. 
\end{Lem} 

\begin{proof} 
Since $s$ is a right homotopy inverse of $\rho$, we have 
$\rho\circ\overline{E}=\rho\circ s\circ a\simeq a$. By definition of $a$, 
we also have $\rho\circ E=a$. If $\ell<p-2$ then by~\cite{T}, $\rho$ is 
an $H$-map, so $\rho\circ(E-\overline{E})\simeq\rho\circ E-\rho\circ\overline{E}$ 
is null homotopic. However, we would also like the statement of the lemma 
to hold for $\ell=p-1$ so we argue without knowing whether 
$\rho\circ (E-\overline{E})$ is null homotopic. 

Define the space $F$ and the map $f$ by the homotopy fibration 
\[\nameddright{F}{f}{\Omega\Sigma A}{\rho}{M(A)}.\] 
By~\cite{CN}, this fibration is modelled in homology by the short exact 
sequence of algebras 
\[0\longrightarrow\nameddright{U[L,L]}{U(g)}{UL}{U(\mathrm{ab})}{UL_{ab}}\longrightarrow 0\] 
where $L$ is the free Lie algebra generated by $\widetilde{H}_{\ast}(A)$, 
$L_{ab}$ is the free abelian Lie algebra (that is, the bracket is identically 
zero) generated by $\widetilde{H}_{\ast}(A)$, $[L,L]$ is the Lie algebra 
kernel of the abelianization map 
\(\namedright{L}{\mathrm{ab}}{L_{ab}}\), 
and $g$ is the inclusion of $[L,L]$ into $L$. So 
$\rho_{\ast}\circ E_{\ast}=\rho_{\ast}\circ\overline{E}_{\ast}$ 
implies by exactness that $E_{\ast}-\overline{E}_{\ast}$ factors 
through $f_{\ast}=U(g)$. But as $g$ is the map sending commutators 
of $L$ into $L$, we obtain $E_{\ast}-\overline{E}_{\ast}=0$ modulo commutators. 
\end{proof} 

The following proposition is the key for decomposing $\Omega(G/H)$. 

\begin{Thm} 
   \label{decomp} 
   Let  
   \(\varphi\colon\namedright{H}{}{G}\) 
   be a homomorphism of Lie groups. Suppose that there is a homotopy 
   commutative diagram 
   \[\diagram 
          \bigvee_{i=1}^{t} A'_{i}\rrto^-{\bigvee_{i=1}^{t} q_{i}}\dto^{j'}  
               & & \bigvee_{i=1}^{t} A_{i}\dto^{j} \\ 
          H\rrto^-{\varphi} & & G 
     \enddiagram\]  
   where $A'_{i},A_{i}\in\mathcal{C}_{p}$ for $1\leq i\leq t$, there are Hopf 
   algebra isomorphisms 
   $H_{\ast}(H)\cong\Lambda(\widetilde{H}_{\ast}(\bigvee_{i=1}^{t} A'_{i}))$ and  
   $H_{\ast}(G)\cong\Lambda(\widetilde{H}_{\ast}(\bigvee_{i=1}^{t} A_{i}))$, 
   and $j',j$ induce the inclusions of the generating sets in homology.  
   Then there is a homotopy commutative diagram 
   \[\diagram 
          \prod_{i=1}^{t} M(A'_{i})\rrto^-{\prod_{i=1}^{t} M(q_{i})}\dto^{e'}  
               & & \prod_{i=1}^{t} M(A_{i})\dto^{e} \\ 
          H\rrto^-{\varphi} & & G 
     \enddiagram\] 
   where $e',e$ are homotopy equivalences. 
\end{Thm} 

\begin{proof} 
First, since $H$ and $G$ are loop spaces, they are homotopy associative 
$H$-spaces, so Theorem~\ref{James} implies that the maps $j'$ and $j$ extend 
to $H$-maps 
\(\bar{j}'\colon\namedright{\Omega\Sigma(\bigvee_{i=1}^{t} A'_{i})}{}{H}\)  
and 
\(\bar{j}\colon\namedright{\Omega\Sigma(\bigvee_{i=1}^{t} A_{i})}{}{G}\). 
Since $\varphi$ is an $H$-map, the uniqueness statement in 
Theorem~\ref{James} implies that there is a homotopy commutative diagram 
\begin{equation} 
  \label{decompdgrm1} 
  \diagram 
      \Omega\Sigma(\bigvee_{i=1}^{t} A'_{i}) 
             \rrto^-{\Omega\Sigma(\bigvee_{i=1}^{t} q_{i})}\dto^{\bar{j}'} 
         & & \Omega\Sigma(\bigvee_{i=1}^{t} A_{i})\dto^{\bar{j}} \\ 
      H\rrto^-{\varphi} & & G. 
  \enddiagram 
\end{equation} 

Second, the inclusion of a wedge summand 
\(\namedright{A_{k}}{}{\bigvee_{i=1}^{t} A_{i}}\) 
induces a map 
\(\namedright{\Omega\Sigma A_{k}}{}{\Omega\Sigma(\bigvee_{i=1}^{t} A_{i})}\). 
The loop multiplication on $\Omega\Sigma(\bigvee_{i=1}^{t} A_{i})$ 
lets us take the product of these maps for $1\leq k\leq t$  
to obtain a map 
\(J\colon\namedright{\prod_{i=1}^{t}\Omega\Sigma A_{k}}{} 
    {\Omega\Sigma(\bigvee_{i=1}^{t} A_{i})}\). 
This construction is natural for a map 
\(\llnamedright{\bigvee_{i=1}^{t} A'_{i}}{\bigvee_{i=1}^{t} q_{i}}{\bigvee_{i=1}^{t} A_{i}}\), 
so we obtain a homotopy commutative diagram 
\begin{equation} 
  \label{decompdgrm2} 
  \diagram 
       \prod_{i=1}^{t}\Omega\Sigma A'_{i}\rrto^-{\prod_{i=1}^{t}\Omega\Sigma q_{i}}\dto^{J'} 
            & &  \prod_{i=1}^{t}\Omega\Sigma A_{i}\dto^{J} \\ 
       \Omega\Sigma(\bigvee_{i=1}^{t} A'_{i})\rrto^-{\Omega\Sigma(\bigvee_{i=1}^{t} q_{i})} 
            & & \Omega\Sigma(\bigvee_{i=1}^{t} A_{i}). 
  \enddiagram 
\end{equation} 

Third, since $A'_{i},A_{i}\in\mathcal{C}_{p}$, by Theorem~\ref{CN}~(b) there 
are maps 
\(s'_{i}\colon\namedright{M(A'_{i})}{}{\Omega\Sigma A'_{i}}\) 
and 
\(s_{i}\colon\namedright{M(A_{i})}{}{\Omega\Sigma A_{i}}\) 
which have left homotopy inverses. The naturality property in 
Theorem~\ref{CN}~(e) then implies that there is a homotopy commutative diagram 
\begin{equation} 
  \label{decompdgrm3} 
  \diagram 
      \prod_{i=1}^{t} M(A'_{i})\rrto^-{\prod_{i=1}^{t} M(q_{i})}\dto^{\prod_{i=1}^{t} s'_{i}} 
          & & \prod_{i=1}^{t} M(A_{i})\dto^{\prod_{i=1}^{t} s_{i}} \\ 
      \prod_{i=1}^{t}\Omega\Sigma A'_{i}\rrto^-{\prod_{i=1}^{t}\Omega\Sigma q_{i}} 
          & & \prod_{i=1}^{t}\Omega\Sigma A_{i}. 
  \enddiagram 
\end{equation} 

Let $e'$ and $e$ be the composites  
\[e'\colon\llnamedddright{\prod_{i=1}^{t} M(A'_{i})}{\prod_{i=1}^{t} s'_{i}} 
        {\prod_{i=1}^{t}\Omega\Sigma A'_{i}}{J'}{\Omega\Sigma(\bigvee_{i=1}^{t} A'_{i})} 
        {\bar{j}'}{H}\] 
\[e\colon\llnamedddright{\prod_{i=1}^{t} M(A_{i})}{\prod_{i=1}^{t} s_{i}} 
        {\prod_{i=1}^{t}\Omega\Sigma A_{i}}{J}{\Omega\Sigma(\bigvee_{i=1}^{t} A_{i})} 
        {\bar{j}}{G}.\] 
Then juxtaposing~(\ref{decompdgrm1}), (\ref{decompdgrm2}) and~(\ref{decompdgrm3}) 
we obtain a homotopy commutative diagram 
\[\diagram 
      \prod_{i=1}^{t} M(A'_{i})\rrto^-{\prod_{i=1}^{t} M(q_{i})}\dto^{e'} 
          & & \prod_{i=1}^{t} M(A_{i})\dto^{e} \\ 
       H\rrto^-{\varphi} & & G. 
  \enddiagram\] 

Finally, we show that $e'$ and $e$ are homotopy equivalences. By Whitehead's 
Theorem, it suffices to show that $e'$ and $e$ induce isomorphisms in homology 
or cohomology. Consider the restriction of $e$ to $\bigvee_{i=1}^{t} A_{i}$, 
that is, consider the composite 
\[\bigvee_{i=1}^{t} A_{i}\stackrel{\bigvee_{i=1}^{t} a_{i}}{\llarrow} 
        \llnamedddright{\prod_{i=1}^{t} M(A_{i})}{\prod_{i=1}^{t} s_{i}} 
        {\prod_{i=1}^{t}\Omega\Sigma A_{i}}{J}{\Omega\Sigma(\bigvee_{i=1}^{t} A_{i})} 
        {\bar{j}}{G}.\] 
By the definition of $a_{i}$ and Theorem~\ref{CN}~(c), $(a_{i})_{\ast}$ is the 
inclusion of the generating set into $H_{\ast}(M(A_{i})$. So if we can show 
that $(e\circ\bigvee_{i=1}^{t} a_{i})_{\ast}$ is the inclusion of the generating 
set into $H_{\ast}(G)$, then $e_{\ast}$ induces an isomorphism on generating 
sets. As $H_{\ast}(M(A))$ and $H_{\ast}(G)$ are primitively generated, dualizing 
to cohomology implies that $e^{\ast}$ is an isomorphism on generating sets. 
Therefore, as $e^{\ast}$ is an algebra map, it is an isomorphism in all degrees. 
The same argument holds for $e'$. 

It remains to show that $(e\circ\bigvee_{i=1}^{t} a_{i})_{\ast}$ is the inclusion 
of the generating set into $H_{\ast}(G)$. By definition of the map $\overline{E}$, we have 
$(\prod_{i=1}^{t} s_{i})\circ\bigvee_{i=1}^{t} a_{i}=\bigvee_{i=1}^{t}\overline{E}_{i}$. 
So by Lemma~\ref{modcomm}, modulo commutators in 
$H_{\ast}(\prod_{i=1}^{t}\Omega\Sigma A_{i})$, this map induces the 
same map in homology as $(\bigvee_{i=1}^{t} E_{i})_{\ast}$. Observe that 
$J$ is a product of $H$-maps and $\bar{j}$ is an $H$-map, so they induce 
algebra maps in homology. Therefore, as $H_{\ast}(G)$ is a commutative algebra, 
$(\bar{j}\circ J)_{\ast}$ sends all commutators in 
$H_{\ast}(\prod_{i=1}^{t}\Omega\Sigma A)$ to zero in $H_{\ast}(G)$. Thus 
$(\bar{j}\circ J\circ(\prod_{i=1}^{t} s_{i})\circ\bigvee_{i=1}^{t} a_{i})_{\ast} 
   =(\bar{j}\circ J\circ(\bigvee_{i=1}^{t} E_{i}))_{\ast}$. 
The left map in this equality is $(e\circ\bigvee_{i=1}^{t} a_{i})_{\ast}$. 
For the right map, by the definitions of $J$ and $\bar{j}$, the composite 
$\bar{j}\circ J\circ(\bigvee_{i=1}^{t} E_{i})\simeq j$. Thus 
$(e\circ\bigvee_{i=1}^{t} a_{i})_{\ast}=j_{\ast}$. By hypothesis, $j_{\ast}$ 
is the inclusion of the generating set into $H_{\ast}(G)$, and hence so 
is $(e\circ\bigvee_{i=1}^{t} a_{i})_{\ast}$.   
\end{proof}         

Let $\fib(M(q_{i}))$ be the homotopy fibre of the map 
\(\lnamedright{M(A'_{i})}{M(q_{i})}{M(A_{i})}\). 
The homotopy commutative diagram in Theorem~\ref{decomp} 
implies that there is a homotopy fibration diagram 
\[\diagram 
          \prod_{i=1}^{t} \fib(M(q_{i}))\rrto\dto^{\epsilon} 
               & & \prod_{i=1}^{t} M(A'_{i})\rrto^-{\prod_{i=1}^{t} M(q_{i})}\dto^{e'}  
               & & \prod_{i=1}^{t} M(A_{i})\dto^{e} \\ 
          \Omega(G/H)\rrto & & H\rrto^-{\varphi} & & G 
     \enddiagram\]   
for some induced map $\epsilon$ of fibres. Since $e',e$ are homotopy 
equivalences, the five-lemma implies that~$\epsilon$ is as well. Thus 
we obtain the following.       

\begin{Cor} 
   \label{decompcor} 
   There is a homotopy equivalence 
   \[\Omega(G/H)\simeq_{p}\prod_{i=1}^{t} \fib(M(q_{i})).\] 
   $\qqed$ 
\end{Cor}

\section{The quasi-$p$-regular case}  

In this section we aim towards Theorem~\ref{qpregAG}, which shows that 
if $H$ and $G$ are both quasi-$p$-regular and satisfy mild restrictions 
on the factors, then the hypotheses of Theorem~\ref{decomp} are satisfied. 
To do this, we first need to study properties of the factors. 

We begin by defining some spaces and maps
following \cite{MNT}. For $n\geq 2$, define 
the space $B(2n-1,2n+2p-3)$ by the homotopy pullback 
\[\diagram 
       S^{2n-1}\rto\ddouble & B(2n-1,2n+2p-3)\rto\dto 
           & S^{2n+2p-3}\dto^{\frac{1}{2}\alpha_{1}(2n)} \\ 
       S^{2n-1}\rto & O(2n+1)/O(2n-1)\rto & S^{2n}. 
  \enddiagram\] 
Notice that 
$H^{\ast}(B(2n-1,2n+2p-3))\cong\Lambda(x_{2n-1},x_{2n+2p-3})$ 
and $\mathcal{P}^{1}(x_{2n-1})=x_{2n+2p-3}$. In particular, 
$B(2n-1,2n+2p-3)$ is a three-cell complex. Let $A(2n-1,2n+2p-3)$ be the 
$(2n+2p-3)$-skeleton of $B(2n-1,2n+2p-3)$ and let 
\[i_{2n-1}\colon\namedright{A(2n-1,2n+2p-3)}{}{B(2n-1,2n+2p-3)}\] 
be the skeletal inclusion. Then $A(2n-1,2n+2p-3)$ is a two-cell complex 
consisting of the bottom two cells in $B(2n-1,2n+2p-3)$. Observe that 
$H^{\ast}(B(2n-1,2n+2p-3))\cong\Lambda(\widetilde{H}^{\ast}(A(2n-1,2n+2p-3))$. 

The space $B(2n-1,2n+2p-3)$ is analogous to $M(A(2n-1,2n+2p-3))$. 
It is introduced in addition to $M(A(2n-1,2n+2p-3))$ because the standard 
homotopy decompositions of Lie groups due to Mimura, Nishida and 
Toda~\cite{MNT} are given in terms of the $B's$, and these will be used 
subsequently as a starting point for producing alternative decomposition 
in terms of $M(A)$'s via Theorem~\ref{decomp}. For now, we note that 
the two are homotopy equivalent provided $p\geq 5$. (If $p=3$ then 
as $A(2n-1,2n+2p-3)$ has two cells, Theorem~\ref{CN} does not apply - 
that is, the space $M(A(2n-1,2n+2p-3))$ does not exist.) 

\begin{Lem} 
   \label{BMequiv} 
   Let $p\geq 5$. If $n\geq 2$ then there is a homotopy equivalence 
   \[M(A(2n-1,2n+2p-3))\simeq_{p} B(2n-1,2n+2p-3).\] 
\end{Lem} 

\begin{proof} 
For simplicity, let $A=A(2n-1,2n+2p-3)$, $B=B(2n-1,2n+2p-3)$ and 
\(i\colon\namedright{A}{}{B}\) 
be $i_{2n-1}$. Since $p\geq 5$, by~\cite{McGibbon} the top cell splits off $\Sigma B$, 
that is, $\Sigma i$ has a left homotopy inverse 
\(t\colon\namedright{\Sigma B}{}{\Sigma A}\). 
Consider the diagram 
\[\diagram 
       A\rto^-{E}\dto^{i} & \Omega\Sigma A\drdouble\dto^{\Omega\Sigma i} & & \\ 
       B\rto^-{E} & \Omega\Sigma B\rto^-{\Omega t} 
           & \Omega\Sigma A\rto^-{\rho} & M(A) 
  \enddiagram\] 
where $\rho$ is the map from Theorem~\ref{CN}. The left square homotopy 
commutes by the naturality of the suspension map $E$ and the triangle homotopy 
commutes since $t$ is a right homotopy inverse of $\Sigma i$. Let 
$e=\rho\circ\Omega\Sigma t\circ E$ be the composition along the 
bottom row. By Theorem~\ref{CN}~(c), $\rho\circ E$ induces the inclusion 
of the generating set in homology, so the homotopy commutativity of the 
preceding diagram implies that $e\circ i$ does as well. But $i$ is 
the inclusion of the $(2n+2p-3)$-skeleton, so it induces the inclusion 
of the generating set in homology. Thus $e_{\ast}$ is a self-map of 
$\Lambda(x_{2n-1},x_{2n+1p-3})$ which is an isomorphim on the 
generating set. As $e$ is a map of spaces, $e_{\ast}$ is a map 
of coalgebras, and any such map satisfies 
$\overline{\Delta}\circ e_{\ast}=(e_{\ast}\otimes e_{\ast})\circ\overline{\Delta}$, 
where $\overline{\Delta}$ is the reduced diagonal. Applying the 
reduced diagonal to the product class $x_{2n-1}\otimes x_{2n+2p-3}$ 
we immediately see that $e_{\ast}$ is also an isomorphism in degree~$4n+2p-4$. 
Thus $e_{\ast}$ is an isomorphism in all degrees and so $e$ is a 
homotopy equivalence. 
\end{proof} 

In what follows we need information about the homotopy sets 
$[A(2n-1,2n+2p-3,S^{2m-1}]$ and $[A(2n-1,2n+2p-3),B(2m-1,2m+2p-3)]$  
for various $n$ and $m$. We do this now, starting by listing some 
known homotopy group calculations. 

\begin{Lem}[Toda~\cite{Toda}] 
  \label{lem:sphere}
  \[
   \pi_{2m-1+t}(S^{2m-1})=
   \begin{cases}
   \Z/p\Z & t=2i(p-1)-1, 1 \le i \le p-1 \\
   \Z/p\Z & t=2i(p-1)-2, m \le i \le p-1 \\
   0 & \text{otherwise for } 1\le t \le 2p(p-1)-3 \end{cases}
  \] 
   $\qqed$ 
\end{Lem}

\begin{Lem}[Mimura-Toda~\cite{Mimura-Toda},Kishimoto~\cite{Kishimoto}] 
  \label{lem:MT}
  \[
   \pi_{3+t}(B(3,2p+1))=
   \begin{cases} \Z/p^2\Z & t=2i(p-1)-1, 2 \le i \le p-1 \\
   \Z_{(p)} & t=2p-2 \\
   0 & \text{otherwise for } 1\le t \le 2p(p-1)-3 \end{cases}
  \]
  \[
   \pi_{2m-1+t}(B(2m-1,2m+2p-3))=
   \begin{cases} 
   \Z/p^2\Z & t=2i(p-1)-1, 2 \le i \le p-1 \\
   \Z/p\Z & t=2i(p-1)-2, m \le i \le p-1 \\
   \Z_{(p)} & t=2p-2 \\
   0 & \text{otherwise for } 1\le t \le 2p(p-1)-3 \end{cases}
  \] 
   $\qqed$ 
\end{Lem} 

\begin{Rem} 
\label{htpygpremark} 
Notice that if $0<t\leq 4p-6$ and $t$ is even then 
$\pi_{2m-1+t}(S^{2m-1})\cong 0$, except in the one case when $m=2$ 
and $\pi_{3+(4p-3)}(S^{3})\cong\Z/p\Z$. Also, if $0<t\leq 4p-6$ 
and $t$ is even then $\pi_{3+t}(B(3,2p+1))\cong 0$ and 
$\pi_{2m-1+t}(B(2m-1,2m+2p-3))\cong 0$. 
\end{Rem} 

\begin{Lem} 
   \label{ABfactor} 
   Let $2\leq m,n\leq p$. Select spaces $A_{m}$ and $B_{n}$ as follows: 
   \[A_{m}\in\{\ast, S^{2m-1}, A(2m-1,2m+2p-3)\}\]  
   \[B_{n}\in\{\ast,S^{2n-1},S^{2n+2p-3}, B(2n-1,2n+2p-3),B(2n+2p-3,2n+4p-5)\}.\]  
   Exclude the case when $A_{m}=A(2p-1,4p-3)$ and $B_{n}=S^{3}$. If $m\neq n$ then 
   $[A_{m},B_{n}]\cong 0$. 
\end{Lem} 

\begin{proof} 
If $A_{m}=\ast$ then we are done. Otherwise, 
the possible dimensions for the nontrivial cells of~$A_{m}$ are $2m-1$ 
and $2m+2p-3$. Observe that $\pi_{2m-1}(B_{n})=\pi_{2n-1+t}(B_{n})$ 
for $t=2m-2n$, and $\pi_{2m+2p-3}(B_{n})=\pi_{2n-1+t'}(B_{n})$ 
for $t'=2m+2p-2n-2$. In particular, both $t$ and $t'$ are even. Also, 
we may assume that $t,t'\geq 0$. As $m\neq n$ we obtain $t>0$, 
and as $2\leq m,n\leq p$, we also obtain $t'>0$. Finally, since 
$2\leq m,n\leq p$ we have $2m+2p\leq 4p$ and $2n\geq 4$. 
Thus $t<4p-6$ and $t'\leq 4p-6$. So by Remark~\ref{htpygpremark}, 
$\pi_{2m-1}(B_{n})\cong 0$ and, as the excluded case in the hypotheses 
rules out obtaining $\pi_{4p-3}(S^{3})$, we also have  
$\pi_{2m+2p-3}(B_{n})\cong 0$. 

Therefore, if $A_{m}=S^{2m-1}$ then $[A_{m},B_{n}]\cong 0$.  
If $A_{m}=A(2m-1,2m+2p-3)$ then the homotopy cofibration 
\(\nameddright{S^{2m-1}}{}{A_{m}}{}{S^{2m+2p-3}}\) 
implies that there is an exact sequence 
\[\nameddright{\pi_{2m+2p-3}(B_{n})}{}{[A_{m},B_{n}]}{}{\pi_{2m-1}(B_{n})}.\]
As the homotopy groups on the left and right are zero we obtain 
$[A_{m},B_{n}]\cong 0$. 
\end{proof} 

Return to Lie groups. Let $G$ be a simply-connected, simple 
compact Lie group which is quasi-$p$-regular. Then by~\cite{MNT} 
there is a homotopy equivalence 
\[G\simeq_{p}\prod_{m=2}^{p} B_{m}\] 
where $B_{m}$ is one of the following: 
\[B_{m}\in\{\ast, S^{2m-1}, B(2m-1,2m+2p-3), S^{2m+2p-3}, B(2m+2p-3,2m+4p-5)\}.\] 
Define $A_{m}$ by the corresponding list: 
\[A_{m}\in\{\ast, S^{2m-1}, A(2m-1,2m+2p-3), S^{2m+2p-3}, A(2m+2p-3,2m+4p-5)\}.\] 
Notice that in each case, $H^{\ast}(B_{m})\cong\Lambda(\widetilde{H}^{\ast}(A_{m}))$. 
Let $j$ be the composite 
\[j\colon\nameddright{\bigvee_{m=2}^{p} A_{m}}{}{\prod_{m=2}^{p} B_{m}}{\simeq}{G}\] 
where the left map is determined by the skeletal inclusion of $A_{m}$ 
into $B_{m}$. Then there is an isomorphism 
$H^{\ast}(G)\cong\Lambda(\widetilde{H}^{\ast}(\bigvee_{m=1}^{p} A_{m}))$ 
for which $j^{\ast}$ is the projection onto the generating set. 

Now suppose that $H=H_{1}\times H_{2}$ where $H_{1}$ and $H_{2}$ 
are simply-connected, simple compact Lie groups which are 
quasi-$p$-regular. (In theory, this could be generalized to a product 
of finitely many such Lie groups, but in practise two factors suffices. 
In fact, it will often be the case that $H_{1}$ is trivial.) By~\cite{MNT} there  
are homotopy equivalences 
\[H_{1}\simeq_{p}\prod_{m=2}^{p} B'_{m,1}\qquad 
      H_{2}\simeq_{p}\prod_{m=2}^{p} B'_{m,2}.\]  
This time we impose a more stringent condition than in the case of $G$. 
We demand that 
\begin{equation} 
  \label{B'restriction} 
  B'_{m,1},B'_{m,2}\in\{\ast, S^{2m-1}, B(2m-1,2m+2p-3)\}. 
\end{equation}   
Let $A'_{m,1},\ A'_{m,2}$ be the corresponding skeleta: 
\[A'_{m,1},A'_{m,2}\in\{\ast, S^{2m-1}, A(2m-1,2m+2p-3)\}.\] 
Let $B'_{m}=B'_{m,1}\times B'_{m,1}$ and $A'_{m}=A'_{m,1}\vee A'_{m,2}$. 
Then there is a homotopy equivalence  
\[H\simeq_{p}\prod_{m=2}^{p} B'_{m}\]
and a map 
\[j'\colon\nameddright{\bigvee_{m=2}^{p} A'_{m}}{}{\prod_{m=2}^{p} B'_{m}}{\simeq}{H}\] 
which induces the inclusion of the generating set in homology. 

\begin{Thm} 
   \label{qpregAG} 
   Let $G$ be a simply-connected, simple compact Lie group, let 
   $H=H_{1}\times H_{2}$ be a product of two such Lie groups, and let 
   \(\varphi\colon\namedright{H}{}{G}\). 
   be a homomorphism. Suppose that both $G$ and~$H$ are quasi-$p$-regular, 
   that the factors of $H$ satisfy~(\ref{B'restriction}), and that if 
   $A'_{m}$ has $A(2p-1,4p-3)$ as a wedge summand then $B_{2}\neq S^{3}$.  
   Then there is a homotopy commutative diagram 
   \[\diagram 
          \bigvee_{m=2}^{p} A'_{m}\rrto^-{\bigvee_{m=2}^{p} q_{m}}\dto^{j'} 
               & & \bigvee_{m=2}^{p} A_{m}\dto^{j} \\ 
          H\rrto^-{\varphi} & & G 
     \enddiagram\] 
   where $j'$ and $j$ induce the inclusions of the generating sets in homology. 
\end{Thm} 

\begin{proof} 
First, consider the composite 
\[\theta_{k}\colon A'_{k}\hookrightarrow\namedddright{\bigvee_{m=2}^{p} A'_{m}}{j}{H} 
        {\varphi}{G}{\simeq}{\prod_{m=2}^{p} B_{m}}.\] 
By Lemma~\ref{ABfactor}, $[A'_{k},B_{m}]\cong 0$ unless $m=k$. Therefore $\theta_{k}$  
factors as the composite 
\[\namedddright{A'_{k}}{\lambda_{k}}{B_{k}}{\mbox{incl}}{\prod_{m=2}^{p} B_{m}} 
          {\simeq}{G}\] 
where $\lambda_{k}$ is the projection of $\theta_{k}$ onto $B_{k}$. 

Next, observe that if $B_{k}\in\{\ast,S^{2m-1},S^{2m+2p-3}\}$ then 
$A_{k}=B_{k}$, so $\lambda_{k}$ factors through the inclusion 
\(\namedright{A_{k}}{}{B_{k}}\) 
(which is the identity map). 
On the other hand, if $B_{k}=B(2m-1,2m+2p-3)$ or $B_{k}=B(2m+2p-3,2m+4p-5)$ 
then as the dimension of $A'_{k}$ is at most $2m+2p-3$, we have~$\lambda_{k}$ 
factoring through the skeletal inclusion 
\(\namedright{A_{k}}{}{B_{k}}\). 
Thus, in any case, $\lambda_{k}$ factors as a composite 
\[\namedright{A'_{k}}{q_{k}}{A_{k}}\hookrightarrow B_{k}\] 
for some map $q_{k}$. 

Putting this together, for each $2\leq k\leq p$ we obtain a homotopy 
commutative diagram 
\[\diagram 
        A'_{k}\rrto^-{q_{k}}\dto^{\mbox{incl}} & & A_{k}\rto^-{\mbox{incl}} 
             &  B_{k}\dto^{\mbox{incl}} \\ 
        \bigvee_{m=2}^{p} A'_{m}\rto^-{j'} & H\rto^-{\varphi} & G\rto^-{\simeq} 
             & \prod_{m=2}^{p} B_{m}.  
  \enddiagram\]  
Taking the wedge sum of these diagrams for $2\leq k\leq p$ and composing 
with the inverse equivalence 
\(\namedright{\prod_{m=2}^{p} B_{m}}{\simeq}{G}\) 
gives the diagram in the statement of the theorem. 
\end{proof} 

\begin{Rem} 
\label{S1variation} 
We will apply Theorem~\ref{qpregAG} in the case when $G/H$ is a 
symmetric space. This requires that we also consider the possibility 
that $H=S^{1}\times H_{2}$. Then $A=S^{1}\vee A'_{2}$, and as $G$ 
is simply-connected, the restriction of the composite 
\(\nameddright{A}{}{H}{\varphi}{G}\) 
to $S^{1}$ is null homotopic. We are left with the composite 
\(\nameddright{A'_{2}}{}{H}{\varphi}{G}\), 
to which Theorem~\ref{qpregAG} applies. We obtain 
a homotopy commutative diagram 
   \[\diagram 
          S^{1}\vee(\bigvee_{m=2}^{p} A'_{m})\rto^-{\mbox{pinch}}\dto^{j'} 
               & \bigvee_{m=2}^{p} A'_{m}\rto^-{\bigvee_{m=2}^{p} q_{m}}  
               & \bigvee_{m=2}^{p} A_{m}\dto^{j} \\ 
          H\rrto^-{\varphi} & & G 
     \enddiagram\] 
\end{Rem}

\section{Identifying the map $q_{m}$ and the homotopy fibre of $M(q_{m})$} 
\label{sec:identify} 

The next step is to try to identify the maps $q_{m}$ in Theorem~\ref{qpregAG} 
and the homotopy fibre of $M(q_{m})$. Since $j',j$ induce the inclusion 
of the generating set in homology, they induce the projection onto the 
generating set in cohomology. Thus $(q_{m})^{\ast}$ is determined by 
the map of indecomposable modules induced by 
\(\namedright{H}{\varphi}{G}\): 
\[Q\varphi^{\ast}\colon\namedright{QH^{\ast}{G}}{}{QH^{\ast}(H)}.\] 
Based on the calculations to come in the subsequent sections, we will 
consider several possibilities for $q_{m}$ with $(q_{m})^{\ast}\neq 0$. 
In Proposition~\ref{Mfibtype} we will show that this cohomology 
information is sufficient to determine the homotopy type of the fibre 
of $M(q_{m})$. 

At this point it is appropriate to notice that if $p=3$ then Theorem~\ref{CN} 
does not apply to the two cell complex $A(2n-1,2n+2p-3)$. That is, the 
space $M(A(2n-1,2n+2p-3))$ does not exist. To avoid this, from now on we 
will assume that all spaces and maps have been localized at a prime $p\geq 5$. 

We begin by listing eight types of maps:  
\[\begin{array}{l} 
     v_{1}\colon\namedright{A_{m}}{}{A_{m}} \\ 
     v_{2}\colon\namedright{S^{2m-1}}{}{A(2m-1,2m+2p-3)} \\ 
     v_{3}\colon\namedright{A(2m-1,2m+2p-3)}{}{S^{2m+2p-3}} \\ 
     v_{4}\colon\namedright{A(2m-1,2m+2p-3)}{}{A(2m+2p-3,2m+4p-5)} \\ 
     v_{5}\colon\namedright{S^{2m-1}\vee S^{2m-1}}{}{S^{2m-1}} \\ 
     v_{6}\colon\namedright{S^{2m-1}\vee S^{2m-1}}{}{A(2m-1,2m+2p-3)} \\ 
     v_{7}\colon\namedright{S^{2m-1}\vee A(2m-1,2m+2p-3)}{}{A(2m-1,2m+2p-3)} \\ 
     v_{8}\colon\namedright{A(2m-1,2m+2p-3)\vee A(2m-1,2m+2p-3)}{} 
               {A(2m-1,2m+2p-3)}. 
  \end{array}\] 
Here, $v_{1}$ is a homotopy equivalence, $v_{2}$ is the inclusion of the bottom 
cell, $v_{3}$is the pinch map to the top cell, $v_{4}$ is the composite of the 
pinch map to the top cell and the inclusion of the bottom cell, $v_{5}$~is 
a homotopy equivalence when restricted to each wedge summand, $v_{6}$ 
is the inclusion of the bottom cell on each wedge summand, $v_{7}$ is the 
inclusion of the bottom cell when restricted to~$S^{2m-1}$ and is a 
homotopy equivalence when restricted to $A_{m}$, and $v_{8}$ is a 
homotopy equivalence when restricted to each copy of $A_{m}$.  

Apply the functor $M$ in Theorem~\ref{CN} to the maps $v_{1}$ 
to $v_{8}$. Using the facts that $M(S^{2n-1})\simeq_{p} S^{2n-1}$ and 
$M(X\vee Y)\simeq_{p} M(X)\times M(Y)$, we obtain maps: 
\[\begin{array}{l} 
     M(v_{1})\colon\namedright{M(A_{m})}{}{M(A_{m})} \\ 
     M(v_{2})\colon\namedright{S^{2m-1}}{}{M(A(2m-1,2m+2p-3))} \\ 
     M(v_{3})\colon\namedright{M(A(2m-1,2m+2p-3))}{}{S^{2m+2p-3}} \\ 
     M(v_{4})\colon\namedright{M(A(2m-1,2m+2p-3))}{}{M(A(2m+2p-3,2m+4p-5))} \\ 
     M(v_{6})\colon\namedright{S^{2m-1}\times S^{2m-1}}{}{S^{2m-1}} \\ 
     M(v_{6})\colon\namedright{S^{2m-1}\times S^{2m-1}}{}{M(A(2m-1,2m+2p-3))} \\ 
     M(v_{7})\colon\namedright{S^{2m-1}\times M(A(2m-1,2m+2p-3))}{} 
               {M(A(2m-1,2m+2p-3))} \\ 
     M(v_{8})\colon\namedright{M(A(2m-1,2m+2p-3))\times M(A(2m-1,2m+2p-3))}{} 
               {M(A(2m-1,2m+2p-3))}. 
  \end{array}\] 
Let $\fib(M(v_{i}))$ be the homotopy fibre of $M(v_{i})$. In Lemma~\ref{Mfibtype} 
we identify the homotopy type of $\fib(M(v_{i})$ for $1\leq i\leq 8$. First we 
need a preliminary lemma, which holds integrally or $p$-locally.  

\begin{Lem} 
   \label{fibrefact} 
   Suppose that there are maps 
   \(\nameddright{X}{f}{Y}{g}{Z}\) 
   where $Y$ and $Z$ are $H$-spaces and $g$ is an $H$-map. Let 
   $h=g\circ f$. If $m$ is the multiplication on $Z$, we obtain a composite 
   \[h\cdot g\colon\nameddright{X\times Y}{h\times g}{Z\times Z}{m}{Z}.\] 
   Let $F$ be the homotopy fibre of $g$. Then the homotopy fibre of $h\cdot g$ 
   is homotopy equivalent to $X\times F$. 
\end{Lem} 

\begin{proof} 
There is a homotopy equivalence 
\(\theta\colon\namedright{X\times Y}{}{X\times Y}\) 
given by sending $(x,y)$ to $(x,\mu(x,y))$ where $\mu$ is the 
multiplication on $Y$. As $g$ is an $H$-map, $h\cdot g$ is homotopic 
to the composite 
\(\psi\colon\namedddright{X\times Y}{\theta}{X\times Y}{\pi_{2}}{Y}{g}{Z}\), 
where $\pi_{2}$ is the projection onto the second factor. The homotopy 
fibre of $\psi$ is clearly $X\times F$, and so this is also the homotopy fibre 
of $h\cdot g$. 
\end{proof}

\begin{Lem} 
   \label{Mfibtype} 
   Let $p\geq 5$. The following hold: 
   \begin{itemize} 
      \item[(1)] $\fib(M(v_{1}))\simeq_{p}\ast$; 
      \item[(2)] $\fib(M(v_{2}))\simeq_{p}\Omega S^{2m+2p-3}$; 
      \item[(3)] $\fib(M(v_{3}))\simeq_{p} S^{2m-1}$; 
      \item[(4)] $\fib(M(v_{4}))\simeq_{p} S^{2m-1}\times\Omega S^{2m+4p-5}$; 
      \item[(5)] $\fib(M(v_{5}))\simeq_{p} S^{2m-1}$; 
      \item[(6)] $\fib(M(v_{6}))\simeq_{p} S^{2m-1}\times\Omega S^{2m+2p-3}$; 
      \item[(7)] $\fib(M(v_{7}))\simeq_{p} S^{2m-1}$; 
      \item[(8)] $\fib(M(v_{8}))\simeq_{p} M(A(2m-1,2m+2p-3))\simeq B(2m-1,2m+2p-3)$. 
   \end{itemize} 
\end{Lem} 

\begin{proof} 
Since $v_{1}$ is a homotopy equivalence, it induces an isomorphism 
in homology, which implies by Theorem~\ref{CN}~(a) that $M(v_{1})$ 
also induces an isomorphism in homology and so is a homotopy 
equivalence. It follows that $\fib(M(v_{1}))\simeq_{p}\ast$, proving part~(1). 

By Theorem~\ref{CN}~(f), the homotopy cofibration 
\(\nameddright{S^{2m-1}}{}{A(2m-1,2m+2p-3)}{}{S^{2m+2p-3}}\) 
induces a homotopy fibration 
\(\nameddright{S^{2m-1}}{}{M(A(2m-1,2m+2q-3))}{}{S^{2m+2p-3}}\). 
We immediately obtain $\fib(M(v_{2}))\simeq_{p}\Omega S^{2m+2p-3}$ 
and $\fib(M(v_{3}))\simeq_{p} S^{2m-1}$, proving parts~(2) and~(3). 

For part~(4), since $v_{4}$ is the composite 
\[\nameddright{A(2m-1,2m+2p-3)}{v_{3}}{S^{2m+2p-3}}{v_{2}} 
     {A(2m+2p-3,2m+4p-5)}\]  
the naturality property in Theorem~\ref{CN} implies that $M(v_{4})$ is 
homotopic to the composite  
\[\llnameddright{M(2m-1,2m+2p-3)}{M(A(v_{3})}{S^{2m+2p-3}}{M(v_{2}))} 
     {M(A(2m+2p-3,2m+4p-5))}.\]  
Further, by~\cite{T}, the maps $M(v_{2})$ and $M(v_{3})$ are $H$-maps 
so we obtain a homotopy pullback of $H$-spaces and $H$-maps 
\begin{equation} 
  \label{alphalift} 
  \diagram 
        S^{2m-1}\rto\ddouble & X\rto\dto & \Omega S^{2m+4p-5}\dto^{\partial} \\ 
        S^{2m-1}\rto & M(A(2m-1,2m+2p-3))\rto^-{M(v_{2})}\dto^{M(v_{4})} 
              & S^{2m+2p-3}\dto^{M(v_{3})} \\ 
        & M(A(2m+2p-3,2m+4p-5))\rdouble & M(A(2m+2p-3,2m+4p-5)) 
  \enddiagram 
\end{equation} 
which defines the $H$-space $X$ and the $H$-map $\partial$. Note 
that $X\simeq_{p}\fib(M(v_{4}))$. 
In general, the attaching map for the $(2n+2p-3)$-cell in 
$M(A(2n-1,2n+2p-3))$ is $\alpha_{1}$, so the fibration connecting map 
\(\partial\colon\namedright{\Omega S^{2n+2p-3}}{}{S^{2n-1}}\) 
satisfies $\partial\circ E\simeq\alpha_{1}$. In our case, after looping~(\ref{alphalift}), 
we obtain a composite of connecting maps 
\(\nameddright{\Omega^{2} S^{2m+4p-5}}{\Omega\partial}{\Omega S^{2m+2p-3}} 
      {\partial'}{S^{2m-1}}\) 
where the homotopy fibre of $\partial'$ is $\Omega M(v_{2})$. We have 
$\partial'\circ\Omega\partial\circ E^{2}\simeq\alpha_{1}\circ\alpha_{1}$, 
which is null homotopic by~\cite{Toda}. Thus $\partial'\circ\Omega\partial\circ E^{2}$ 
lifts through $\Omega M(v_{2})$. Adjointing, this implies that 
$\partial\circ E$ lifts through $M(v_{2})$ to a map 
\(\lambda\colon\namedright{S^{2m+4p-6}}{}{M(A(2m-1,2m+2p-3))}\). 
By~\cite{T}, $M(A(2m-1,2m+2p-3))$ is homotopy associative, so by 
Theorem~\ref{James}, $\lambda$ extends to an $H$-map 
\[\gamma\colon\namedright{\Omega S^{2m+4p-5}}{}{M(A(2m-1,2m+2p-3))},\] 
and as $M(v_{2})$ is an $H$-map, the uniqueness property of 
Theorem~\ref{James} implies that $M(v_{2})\circ\gamma\simeq\partial$. 
The pullback property of $X$ therefore implies that $\gamma$ pulls back 
to a map 
\(\namedright{\Omega S^{2m+4p-5}}{}{X}\) 
which is a right homotopy inverse for 
\(\namedright{X}{}{\Omega S^{2m+4p-5}}\). 
Since $X$ is an $H$-space, this section implies that there is a homotopy 
equivalence $X\simeq_{p} S^{2m-1}\times\Omega S^{2m+4p-5}$. 

Parts~(5) through~(8) are all special cases of Lemma~\ref{fibrefact}. 
\end{proof} 

Next, we aim to show that if $(q_{m})^{\ast}\neq 0$ in cohomology then $q_{m}$ 
can be described in terms of the maps $v_{1}$ to $v_{8}$. 

\begin{Lem} 
   \label{qID1} 
   Let 
   \(q_{m}\colon\namedright{A'_{m}}{}{A_{m}}\) 
   be a map as in Theorem~\ref{qpregAG} and suppose that, in cohomology, 
   $(q_{m})^{\ast}\neq 0$. Write $u$ for an arbitrary unit in $\mathbb{Z}_{(p)}$. 
   Then the following hold: 
   \begin{itemize} 
      \item[(1)] if $A'_{m}=A_{m}$ then $q_{m}$ is a homotopy equivalence; 
      \item[(2)] if $A'_{m}=S^{2m-1}$ and $A_{m}=A(2m-1,2m+2p-3)$ then 
                      $q_{m}\simeq u\cdot v_{2}$; 
      \item[(3)] if $A'_{m}=A(2m-1,2m+2p-3)$ and $A_{m}=S^{2m+2p-3}$ 
                      then $q_{m}\simeq u\cdot v_{3}$; 
      \item[(4)] if $A'_{m}=A(2m-1,2m+2p-3)$ and $A_{m}=A(2m+2p-3,2m+4p-5)$ 
                      then $q_{m}\simeq u\cdot v_{4}$. 
   \end{itemize} 
\end{Lem} 

\begin{proof} 
For part~(1), if $A'_{m}=A_{m}$ equals $\ast$ or $S^{2m-1}$ then the 
assertion is clear. If they both equal $A(2m-1,2m+2p-3)$ then recall 
that $H^{\ast}(A_{m})=\Z/p\Z\{x_{2m-1},x_{2m+2p-3}\}$ and 
$\mathcal{P}^{1}(x_{2m-1})=x_{2m+2p-3}$. This Steenrod operation 
implies that if $(q_{m})^{\ast}$ is nonzero on either generator then 
it is nonzero on both. Consequently, $(q_{m})^{\ast}$ is an isomorphism 
and so $q_{m}$ is a homotopy equivalence. 

Part~(2) is a consequence of the Hurewicz Theorem. 

For parts~(3) and (4), observe that there is a homotopy cofibration sequence 
\[\llnamedddright{S^{2m-1}}{i}{A(2m-1,2m+2p-3)}{q}{S^{2m+2p-3}} 
     {\alpha_{1}(2m)}{S^{2m}}\] 
where $i$ is the inclusion of the bottom cell and $q$ is the pinch map onto 
the top cell. For any space~$X$, we obtain an induced exact sequence 
\[\namedddright{\pi_{2m}(X)}{}{\pi_{2m+2p-3}(X)}{q^{\ast}}{[A(2m-1,2m+2p-3),X]} 
       {i_{\ast}}{\pi_{2m-1}(X)}.\] 
Taking $X=S^{2m+2p-3}$ or $X=A(2m+2p-3,2m+4p-5)$, by connectivity 
$\pi_{2m}(X)\cong\pi_{2m-1}(X)\cong 0$, so $q^{\ast}$ is an isomorphism.  
The Hurewicz Theorem implies in either case that $\pi_{2m+2p-3}(X)$ is isomorphic  
to $H^{\ast}(X)$. Therefore, in both cases, the homotopy class of~$q_{m}$ 
is determined by its image in cohomology, and the assertions follow. 
\end{proof} 

Arguing as for Lemma~\ref{qID1} we also obtain the following. 
 
\begin{Lem} 
   \label{qID2} 
   Let 
   \(q_{m}\colon\namedright{A'_{m,1}\vee A'_{m,2}}{}{A_{m}}\) 
   be a map as in Theorem~\ref{qpregAG} and suppose that, in cohomology, 
   $(q_{m})^{\ast}\neq 0$ when projected to either $H^{\ast}(A'_{m,1})$ or 
   $H^{\ast}(A'_{m,2})$. Write $u,u'$ for arbitrary units in $\mathbb{Z}_{(p)}$. 
   Then the following hold: 
   \begin{itemize} 
      \item[(5)] if $A'_{m,1}=A'_{m,2}=S^{2m-1}$ and $A_{m}=S^{2m-1}$ 
                      then $q_{m}\simeq u\vee u'$ is a wedge sum of homotopy 
                      equivalences; 
      \item[(6)] if $A'_{m,1}=A'_{m,2}=S^{2m-1}$ and $A_{m}=A(2m-1,2m+2p-3)$ 
                      then $q_{m}\simeq u\cdot v_{2}\vee u'\cdot v_{2}$; 
      \item[(7)] if $A'_{m,1}=S^{2m-1}$, $A'_{m,2}=A(2m-1,2mp+2p-3)$ and 
                      $A_{m}=A(2m-1,2m+2p-3)$ then $q_{m}\simeq u\cdot v_{2}\vee e'$ 
                      where $e'$ is a homotopy equivalence; 
      \item[(8)] if $A'_{m,1}=A'_{m,2}=A(2m-1,2mp+2p-3)$ and 
                      $A_{m}=A(2m-1,2m+2p-3)$ then $q_{m}\simeq e\vee e'$ 
                      where $e,e'$ are homotopy equivalences. 
   \end{itemize} 
   $\qqed$ 
\end{Lem} 

Lemmas~\ref{qID1} and~\ref{qID2} identify $q_{m}$ in terms of the maps $v_{i}$, 
up to multiplication by units in $\mathbb{Z}_{(p)}$ or homotopy equivalences. 
Thus $M(q_{m})$ can similarly be written in terms of the maps $M(v_{i})$. 
As multiplication by a unit in $\mathbb{Z}_{(p)}$ or composition with a homotopy 
equivalence does not affect the homotopy type of the fibre, the homotopy fibre 
of $M(q_{m})$ has the same homotopy type as the homotopy fibre of the 
corresponding $M(v_{i})$'s. So Lemma~\ref{Mfibtype} implies the following. 

\begin{Prop} 
   \label{MqID} 
   Let $p\geq 5$ and let 
   \(q_{m}\colon\namedright{A'_{m}}{}{A_{m}}\) 
   be a map as in Theorem~\ref{qpregAG}. If $(q_{m})^{\ast}\neq 0$, then - 
   listing cases as in Lemmas~\ref{qID1} and~\ref{qID2} - the homotopy 
   fibre of $M(q_{m})$ is as follows: 
   \begin{itemize} 
      \item[(1)] $\fib(M(q_{m}))\simeq_{p}\ast$; 
      \item[(2)] $\fib(M(q_{m}))\simeq_{p}\Omega S^{2m+2p-3}$; 
      \item[(3)] $\fib(M(q_{m}))\simeq_{p} S^{2m-1}$; 
      \item[(4)] $\fib(M(q_{m}))\simeq_{p} S^{2m-1}\times\Omega S^{2m+4p-5}$; 
      \item[(5)] $\fib(M(q_{m}))\simeq_{p} S^{2m-1}$; 
      \item[(6)] $\fib(M(q_{m}))\simeq_{p} S^{2m-1}\times\Omega S^{2m+2p-3}$; 
      \item[(7)] $\fib(M(q_{m}))\simeq_{p} S^{2m-1}$; 
      \item[(8)] $\fib(M(q_{m}))\simeq_{p} M(A(2m-1,2m+2p-3))\simeq_{p} B(2m-1,2m+2p-3)$. 
   \end{itemize} 
   $\qqed$ 
\end{Prop}

\section{Case by case analysis}
In this section, we give homotopy decompositions of $\Omega(G/H)$ 
when $G$ is quasi-$p$-regular using a case by case analysis.
Note that when $G$ is quasi-$p$-regular, $H$ is automatically so
by the classification of the symmetric space.
The classical cases are considered first, followed by the exceptional cases. 

\subsection{Classical cases}
The following homotopy decompositions for quasi-$p$-regular classical 
Lie groups are due to Mimura and Toda~\cite{Mimura-Toda}. 
\begin{Thm}\label{MTdecomp} 
For an odd prime $p$, there are homotopy equivalences: 
\[
\begin{array}{|c|c|c|}
\hline
G & p \text{ (odd) } & \\
\hline
\SU(n) & p>n/2 &
\displaystyle
\prod_{i=2}^{n-p+1}
B(2i-1,2i+2p-3) \times 
\prod_{j=\max(2,n-p+2)}^{\min(n,p)}
S^{2j-1}
 \\
\hline
\SO(2n+1) & p>n &
\displaystyle
\prod_{i=1}^{n-\frac{p-1}{2}}B(4i-1,4i+2p-3) \times 
\prod_{j=n-\frac{p-3}{2}}^{\min(n,\frac{p-1}{2})} S^{4j-1}\\
\hline
\Sp(n) & p>n &
\displaystyle
\prod_{i=1}^{n-\frac{p-1}{2}}B(4i-1,4i+2p-3) \times 
\prod_{j=n-\frac{p-3}{2}}^{\min(n,\frac{p-1}{2})} S^{4j-1} \\
\hline
\SO(2n) & p>n-1 &
\displaystyle
\prod_{i=1}^{n-\frac{p+1}{2}}B(4i-1,4i+2p-3) \times 
\prod_{j=n-\frac{p-1}{2}}^{\min(n-1,\frac{p-1}{2})} S^{4j-1} \times
S^{2n-1} \\
\hline
\end{array}
\] 
$\qqed$ 
\end{Thm} 

We will also use the following homotopy decompositions, due to Harris~\cite{Harris3}.
\begin{Thm}[\cite{Harris3}]
\label{harris-decomp}
For an odd prime $p$, there are homotopy equivalences: 
\begin{align*}
 \SU(2n)  & \simeq_p \Sp(n) \times \SU(2n)/\Sp(n) \\
 \SU(2n+1) & \simeq_p \Spin(2n+1) \times \SU(2n+1)/\Spin(2n+1) \\
 \SO(2n+1) & \simeq_p \Spin(2n+1) \simeq_p \Sp(n) \\
 \SO(2n) & \simeq_p \Spin(2n) \simeq_p \Spin(2n-1) \times S^{2n-1}
\end{align*} 
$\qqed$ 
\end{Thm} 

For expositional purposes, the $AIII$ case is examined first. 
\medskip 

\subsubsection{Type $AIII$}
Assume that $2m \le n$.
Observe that $\SU(n)/\SU(n-m) =U(n)/U(n-m)$.
Since the upper-left inclusion and the lower-right inclusions for $U(n)$
are conjugate and thus homotopic,
the inclusion $U(m)\times U(n-m) \hookrightarrow U(n)$ is homotopic to
\[
 U(m)\times U(n-m) \stackrel{\iota_m \times \iota_{m-n}}{\hookrightarrow}
  U(n) \times U(n) \xrightarrow{\mu} U(n),
\]
where $\iota_m: U(m) \hookrightarrow U(n)$ and $\iota_{n-m}: U(n-m) \hookrightarrow U(n)$
are the upper-left inclusions.
By Lemma \ref{fibrefact}, for $m \le n-m$ there is an integral homotopy equivalence 
\[\Omega (U(n)/U(n-m)\times U(m))\simeq U(m) \times \Omega (\SU(n)/\SU(n-m)).\] 

By Theorem~\ref{MTdecomp}, there are homotopy equivalences 
\begin{align*}
  \SU(n-m) &=\prod_{i=2}^{n-m-p+1} B(2i-1,2i+2p-3)\times 
        \prod_{j=n-m-p+2}^{\min(p,n-m)} S^{2j-1}\\ 
  \SU(n) &= \prod_{i=2}^{n-p+1} B(2i-1,2i+2p-3)\times 
         \prod_{j=n-p+2}^{\min(p,n)} S^{2j-1}.
\end{align*} 
So if we define spaces $A'_{i}$ and $A_{i}$ for $i\leq 2\leq p$ by 
\begin{align*}
\bigvee_{i=2}^{p} A'_i &=\bigvee_{i=2}^{n-m-p+1} A(2i-1,2i+2p-3)\vee 
        \bigvee_{j=n-m-p+2}^{\min(p,n-m)} S^{2j-1}\\ 
\bigvee_{i=2}^{p} A_i &= \bigvee_{i=2}^{n-p+1} A(2i-1,2i+2p-3)\vee 
         \bigvee_{j=n-p+2}^{\min(p,n)} S^{2j-1}
\end{align*} 
then by Theorem~\ref{qpregAG} there is a homotopy commutative diagram 
\[
\xymatrix{
 \bigvee_{i=2}^{p} A'_i \ar[r]^{\bigvee_{i=2}^{p} q_{i}} \ar[d]  
        & \bigvee_{i=2}^{p} A_i  \ar[d]  \\
 \SU(n-m) \ar[r]^-{\varphi} & \SU(n).} 
\] 
In each case, since $\varphi^{\ast}$ is a projection, each $(q_{i})^{\ast}$ 
is an epimorphism. 
So by Proposition~\ref{MqID} and Corollary~\ref{decompcor} we have 
\[
\Omega (\SU(n)/\SU(n-m))\simeq_{p}\prod_{i=2}^{p}\fib(M(q_{i})) \simeq_p
\prod_{j=n-m+1}^{n}\!\! 
\Omega S^{2j-1}.
\] 
Thus, for $p>n/2$, we obtain 
\begin{multline*}
 \Omega (U(n)/U(m)\times U(n-m)) \simeq_p U(m) \times \Omega (\SU(n)/\SU(n-m)) \\
  \qquad \simeq_p 
\prod_{j=1}^{m} S^{2j-1}\! 
\times\!\!\! \!\!
\prod_{j=n-m+1}^{n}\!\! 
\Omega S^{2j-1}.
\end{multline*}

\begin{Rem} 
Using a different approach, a homotopy decomposition for 
$\Omega(\SU(n)/\SU(n-m))$ is obtained in~\cite{Beben,Grbic-Zhao} which 
holds for $n\leq (p-1)(p-2)$. This range includes the quasi-$p$-regular cases 
and more. However, those methods do not extend to exceptional cases while 
ours do, so the argument above was given in detail for the sake of illustrating 
our approach. 
\end{Rem}

\subsubsection{Type $CII$}
Assume $2m \le n$.
Similar to the type $AIII$ case, for $n<p$ we have 
\begin{multline*}
\Omega(\Sp(n)/(\Sp(m)\times \Sp(n-m))) \simeq_p 
\Sp(m) \times \Omega(\Sp(n)/\Sp(n-m)) 
 \\ \qquad
 \simeq_p
\prod_{j=1}^{m} S^{4j-1}
\times\!\! \!\!\!\!
\prod_{j=n-m+1}^{n}
\Omega S^{4j-1}. 
\end{multline*}

\subsubsection{Type $BDI$}
Similar to the type $AIII$ case, we have 
\[
\Omega(\SO(n)/(\SO(m)\times \SO(n-m))) \simeq_p 
\SO(m)\times \Omega (\SO(n)/\SO(n-m)),
\]
where $2m \le n$. By Theorem~\ref{harris-decomp}, for $p$ odd there are 
homotopy equivalences $\SO(2k+1)\simeq_{p}\Sp(k)$ and 
$\SO(2k+2)\simeq_{p}\Sp(k)\times S^{2k+1}$. Therefore, we obtain 
homotopy equivalences:\\
\begin{align*}
\Omega(\SO(2n+1)/\SO(2(n-m)+1)) &\simeq_p 
\Omega(\Sp(n)/\Sp(n-m)) \\
\Omega(\SO(2n+1)/\SO(2(n-m)+2)) &\simeq_p 
S^{2(n-m)+1} \times \Omega(\Sp(n)/\Sp(n-m))  \\
\Omega(\SO(2n+2)/\SO(2(n-m)+1)) &\simeq_p 
\Omega S^{2n+1} \times \Omega(\Sp(n)/\Sp(n-m)) \\
\Omega(\SO(2n+2)/\SO(2(n-m)+2)) & \simeq_p 
S^{2(n-m)+1} \times \Omega S^{2n+1} \times \Omega(\Sp(n)/\Sp(n-m)).
\end{align*} 
Complete decompositions are now obtained from the $CII$ case.

\subsubsection{Types $AI,AII$} 
Homotopy decompositions of $\SU(2n)/\Sp(n)$ and $\SU(2n+1)/\SO(2n+1)$ 
are given in \cite[Thm 4.1]{MNT} as sub-decompositions of $\SU(n)$. 
\begin{align*}
\SU(2n)/\Sp(n) &\simeq_p
\prod_{i=1}^{n-\frac{p+1}{2}}\!\! 
B(4i+1,4i+2p-1) \times 
\prod_{j=\min(1,n-\frac{p-1}{2})}^{\min(n-1, \frac{p-1}{2})}\!\!\!\! 
S^{4j+1} 
\quad (p>n) \\
\SU(2n+1)/\SO(2n+1) &\simeq_p 
\prod_{i=1}^{n-\frac{p-1}{2}}\!\! 
B(4i+1,4i+2p-1) \times 
\prod_{j=\min(1,n-\frac{p-3}{2})}^{\min(n, \frac{p-1}{2})}\!\!\!\! 
S^{4j+1} \quad (p>n).
\end{align*}

For $\SU(2n)/\SO(2n)$, by~\cite[Theorem 6.7]{Mimura-Toda2}, 
\(Q\varphi^{\ast}\colon\namedright{QH^{t}(\SU(2m))}{}{QH^{t}(\SO(2m))}\) 
is nontrivial for $t\in\{3,7,\ldots,4m-5\}$. So arguing as in the $AIII$ 
case, we obtain a homotopy equivalence 
\[
\Omega \SU(2n)/\SO(2n) \simeq_p 
\Omega S^{2n}\times\prod_{i=1}^{n-\frac{p+1}{2}}\!\! 
\Omega B(4i+1,4i+2p-1) \times 
\prod_{j=\min(1,n-\frac{p-1}{2})}^{\min(n-1, \frac{p-1}{2})}\!\!\!\! 
\Omega S^{4j+1} \quad (p>n).
\]

\subsubsection{Types $CI,DIII$}
For the type $CI$ case of $\Sp(n)/U(n)$,
$\Sp(n)$ is quasi regular when $p>n$ and then
$U(n)\simeq_{p}\prod_{i=1}^n S^{2i-1}$. By~\cite[Theorem 5.8]{Mimura-Toda2}, 
\(Q\varphi^{\ast}\colon\namedright{QH^{t}(\Sp(n))}{}{QH^{t}(U(n))}\) 
is nontrivial for $t\in\{3,7,\ldots,4[n/2]\}$. So arguing as in the $AIII$ case we 
obtain a homotopy equivalence 
\[
\Omega (\Sp(n)/U(n)) \simeq_p 
\prod_{j=0}^{\left[\frac{n-1}{2}\right]}\!\! 
S^{4j+1} \times
\prod_{j=\left[\frac{n+2}{2}\right]}^{n}\!\! 
\Omega S^{4j-1},
\quad (p>n).
\]

For the type $DIII$ case of $\SO(2n)/U(n)$, we can reduce it to a type $CII$ case by
\[
\SO(2n)/U(n)=\SO(2n-1)/U(n-1) 
\simeq_p \Sp(n-1)/U(n-1).
\]


Summarising the results for classical cases, we have the following. 
\begin{Thm}\label{thm:classical}
For $p\geq 5$, there are homotopy equivalences: \\ 
\text{ }\\
\scalebox{0.85}{
\begin{minipage}{9cm}
\[
\begin{array}{|c|c||c|c|}
\hline
\text{Type} & G/H & p\geq 5 & \text{Homotopy type of }  \Omega(G/H)
\\ \hline 
AI &\SU(2n+1)/\SO(2n+1) & p>n &
\displaystyle
\prod_{i=1}^{n-\frac{p-1}{2}}\!\! 
\Omega B(4i+1,4i+2p-1) \times 
\prod_{j=\min(1,n-\frac{p-3}{2})}^{\min(n, \frac{p-1}{2})} \Omega S^{4j+1}
\\ 
& \SU(4n+2)/\SO(4n+2) & p=2n+1 &
\displaystyle
\prod_{i=1}^{n-1}
\Omega B(4i+1,4i+2p-1) 
\times \Omega S^{8n+1}
\times \Omega S^{8n+3}
\\
 & \SU(2n)/\SO(2n) & p>2n &
\displaystyle
\Omega S^{2n}\times\prod_{i=1}^{n-\frac{p+1}{2}}\!\! 
\Omega B(4i+1,4i+2p-1) \times 
\prod_{j=\min(1,n-\frac{p-1}{2})}^{\min(n-1, \frac{p-1}{2})}\!\!\!\! 
\Omega S^{4j+1}
\\ \hline
AII & \SU(2n)/\Sp(n) & p>n &
\displaystyle
\prod_{i=1}^{n-\frac{p+1}{2}}\!\! 
\Omega B(4i+1,4i+2p-1) \times 
\prod_{j=\max(1,n-\frac{p-1}{2})}^{\min(n-1, \frac{p-1}{2})}\!\!\!\! 
\Omega S^{4j+1}
\\ \hline
AIII & \dfrac{U(n)}{U(m)\times U(n-m)}^\dag & p>n/2 &
\displaystyle
\prod_{j=1}^{m} S^{2j-1}\! 
\times\!\!
\prod_{j=n-m+1}^{n}\!\! 
\Omega S^{2j-1}
\\ \hline
BDI & \dfrac{\SO(2n+1)}{\SO(2m)\times \SO(2(n-m)+1)}^\dag & p>n &
\displaystyle
\prod_{j=1}^{m-1} S^{4j-1}
\times S^{2m-1}
\times\!\!
\prod_{j=n-m+1}^{n}
\Omega S^{4j-1}
\\
 & \dfrac{\SO(2n+1)}{\SO(2m-1)\times \SO(2(n-m)+2)}^\ddag & p>n &
\displaystyle
\prod_{j=1}^{m-1} S^{4j-1}
\times S^{2(n-m)+1}
\times\!\!
\prod_{j=n-m+1}^{n}
\Omega S^{4j-1}
\\
 & \dfrac{\SO(2n+2)}{\SO(2m+1)\times \SO(2(n-m)+1)}^\dag & p>n &
\displaystyle
\prod_{j=1}^{m} S^{4j-1}
\times \Omega S^{2n+1} 
\times\!\!
\prod_{j=n-m+1}^{n}
\Omega S^{4j-1}
\\
 & \dfrac{\SO(2n+2)}{\SO(2m)\times \SO(2(n-m)+2)}^\ddag & p>n-1 &
\displaystyle
\prod_{j=1}^{m-1} S^{4j-1}
\times S^{2m-1} \times
S^{2(n-m)+1} \times \Omega S^{2n+1}
\times\!\!
\prod_{j=n-m+1}^{n}
\Omega S^{4j-1}
\\ \hline
CI & \Sp(n)/U(n) & p>n &
\displaystyle
\prod_{j=0}^{\left[\frac{n-1}{2}\right]}\!\! 
S^{4j+1} \times
\prod_{j=\left[\frac{n+2}{2}\right]}^{n}\!\! 
\Omega S^{4j-1}
\\ \hline
CII & \dfrac{\Sp(n)}{\Sp(m)\times \Sp(n-m)}^\dag & p>n &
\displaystyle
\prod_{j=1}^{m} S^{4j-1}
\times\!\! \!\!\!\!
\prod_{j=n-m+1}^{n}
\Omega S^{4j-1}
\\ \hline
DIII & \SO(2n)/U(n) & p>n-1 &
\displaystyle
\prod_{j=0}^{\left[\frac{n-2}{2}\right]}\!\! 
S^{4j+1} \times
\prod_{j=\left[\frac{n+1}{2}\right]}^{n-1}\!\! 
\Omega S^{4j-1}
\\ \hline
\end{array}
\]
\end{minipage}
}
\begin{itemize}
\item for $\dag$, we assume $2m\le n$
\item for $\ddag$, we assume $2m\le n+1$
\end{itemize} 
$\qqed$ 
\end{Thm}

\begin{Rem} 
\label{Terzicremark1} 
Terzi\'{c}'s computation of the rational homotopy groups of classical 
symmetric spaces in~\cite{Terzic} can be reproduced from the decompositions 
above. Our list corrects a typo in her description of the rational homotopy type of 
$\SO(2n)/U(n)$. See also Remark~\ref{Terzicremark2} for the exceptional cases. 
\end{Rem} 

\begin{Rem} 
Mimura~\cite{Mimura2} showed that the homotopy decompositions 
for types $AI$ and $AII$ deloop. He also showed that these cases 
hold for $p=3$ as well, and the $AII$ case can be strengthened to hold 
for $p\geq n$. 
\end{Rem}

\subsection{Exceptional cases} 
The following homotopy decompositions for quasi-$p$-regular exceptional 
Lie groups are due to Mimura and Toda~\cite{Mimura-Toda}. 

\begin{Thm}\label{MTexceptional} 
For an odd prime $p$, there are homotopy equivalences: 
\[
\begin{array}{|c|c|c|}
\hline
G & p & \\
\hline
G_2 & 5 & B(3,11) \\
& \ge 7 & S^3 \times S^{11} \\
\hline
F_4 & 5 & B(3,11)\times B(15,23)  \\
& 7 & B(3,15)\times B(11,23) \\
& 11 & B(3,23)\times S^{11} \times S^{15} \\
& \ge 13 & S^3 \times S^{11} \times S^{15} \times S^{23}\\
\hline
E_6 & 5 & B(3,11)\times B(9,17) \times B(15,23) \\
& 7 & B(3,15)\times B(11,23) \times S^9 \times S^{17} \\
& 11 & B(3,23)\times S^9 \times S^{11} \times S^{15} \times S^{17}  \\
& \ge 13 & S^3\times S^9 \times S^{11} \times S^{15} \times S^{17} \times S^{23} \\
\hline
E_7 & 11 & B(3,23) \times B(15,35) \times S^{11} \times S^{19} \times S^{27} \\
& 13 & B(3,27) \times B(11,35) \times S^{15} \times S^{19} \times S^{23}  \\
& 17 & B(3,35) \times S^{11} \times S^{15} \times S^{19} \times S^{23} \times S^{27}  \\
& \ge 19 & S^3 \times S^{11} \times S^{15} \times S^{19} \times S^{23} \times S^{27}\times S^{35}  \\
\hline
E_8 & 11 & B(3,23) \times B(15,35) \times B(27,47) \times B(39,59) \\
& 13 & B(3,27) \times B(15,39) \times B(23,47) \times B(35,59) \\
& 17 & B(3,35) \times B(15,47) \times B(27,59) \times S^{23} \times S^{39}  \\
& 19 & B(3,39) \times B(23,59) \times S^{15} \times S^{27} \times S^{35} \times S^{47}  \\
& 23 & B(3,47) \times B(15,59) \times S^{23} \times S^{27} \times S^{35} \times S^{39} \\
& 29 & B(3,59) \times S^{15} \times S^{23} \times S^{27} \times S^{35} \times S^{39} \times S^{47}  \\
& \ge 31 & S^3 \times S^{15} \times S^{23} \times S^{27} \times S^{35} \times S^{39} \times S^{47} \times S^{59}\\
\hline
\end{array}
\] 
$\qqed$ 
\end{Thm} 

In analyzing the loop space of an exceptional symmetric space corresponding 
to a map 
\(\varphi\colon\namedright{H}{}{G}\) 
between quasi-$p$-regular Lie groups, we will use the following strategy. 
\medskip 

\noindent\textbf{Strategy}: 
\begin{itemize} 
   \item[(1)] Use Theorem~\ref{decomp} to replace $\namedright{H}{\varphi}{G}$ by 
                   \(\lllnamedright{\prod_{m=2}^{p} M(A'_{m})} 
                           {\prod_{m=2}^{p} M(q_{m})}{\prod_{m=2}^{p} M(A_{m})}\). 
   \item[(2)] Determine those $q_{m}$ which are nontrivial in cohomology 
                   via the induced map of indecomposable modules, 
                   \(Q\varphi^{\ast}\colon\namedright{QH^{\ast}{(G)}}{}{QH^{\ast}(H)}\). 
   \item[(3)] Observe that the remaining maps  
                   \(q_{m}\colon\namedright{A'_{m}}{}{A_{m}}\) 
                   are trivial because either $A'_{m}$ or $A_{m}$ is trivial. 
   \item[(4)] Deduce the homotopy fibre of $M(q_{m})$ from Proposition~\ref{MqID} 
                   or from the fact that $M(q_{m})$ is trivial. 
   \item[(5)] Use Corollary~\ref{decompcor} to obtain 
                   $\Omega(G/H)\simeq\prod_{m=2}^{p}\fib(M(q_{m}))$.  
\end{itemize} 
\medskip 

\subsubsection{Type $G$}
Recall that $\SO(4) \simeq_p S^3 \times S^3$ for $p\ge 5$.
For $p=5$, by Theorem~\ref{qpregAG} there is a homotopy 
commutative diagram 
\[
\xymatrix{
 S^3 \vee S^3 \ar[r] \ar[d] &
    A(3,11) \ar[d] \\
 \SO(4) \ar[r]^-{\varphi} & G_2.}
\]
Since
\[
 \varphi^*: QH^3(G_2;\F_p) \to QH^3(\SO(4);\F_p) \cong QH^{3}(S^{3}\times S^{3};\F_p)
\]
is non-trivial, 
Proposition~\ref{MqID} implies that there is a homotopy equivalence 
\[
 \Omega(G_2/\SO(4)) \simeq_p S^3 \times \Omega S^{11} \quad (p= 5).
\]
For $p>5$, the space $A(3,11)$ is replaced by $S^{3}\vee S^{11}$ and 
arguing as in the $p=5$ case we obtain 
\[
 \Omega(G_2/\SO(4)) \simeq_p S^3 \times \Omega S^{11} \quad (p> 5).
\]

\subsubsection{Type $FI$}
By Theorem~\ref{MTdecomp} there are homotopy equivalences 
\[
\SU(2)\cdot \Sp(3) \simeq_p \begin{cases}
S^3 \times B(3,11) \times S^7 & (p=5) \\
S^3 \times S^3 \times S^7 \times S^{11}& (p>5).
\end{cases} 
\]
It is shown in \cite{Ishitoya-Toda} that
\[
 H^*(FI; \F_p) = \F_p [f_4, f_8]/(r_{16}, r_{24}), \quad (p \ge 5)
\] 
for some relations $r_{16}, r_{24}$ in degrees $16$ and $24$ respectively. 
Thus  
\[
Q\varphi^*: QH^m(F_4; \F_p) \to QH^m(\SU(2)\cdot\Sp(3); \F_p) 
\]
is non-trivial for $m\in\{3,11\}$ and $p \ge 5$.
When $p=5$, by Theorem~\ref{qpregAG} there is a homotopy commutative 
diagram 
\[
\xymatrix{
 S^3 \vee A(3,11) \vee S^7 \ar[r] \ar[d] & A(3,11) \vee A(15,23) \ar[d] \\
 \SU(2) \cdot \Sp(3) \ar[r]^-{\varphi}  &  F_4. 
}
\]
Proposition \ref{MqID} therefore implies that there is a homotopy equivalence 
\[
 \Omega FI \simeq_5 S^3 \times S^7 \times \Omega B(15,23).
\] 

For $p>5$, arguing similarly we obtain 
\[
 \Omega FI \simeq_p S^3 \times S^7 \times \Omega S^{15} \times \Omega S^{23}.
\]


%
\subsubsection{Type $FII$} 
By Theorem~\ref{MTdecomp} there are homotopy equivalences 
\[
\Spin(9) \simeq_p \begin{cases}
B(3,11) \times B(7,15) & (p=5) \\
B(3,15) \times S^7 \times S^{11} & (p=7) \\
S^3 \times S^7 \times S^{11} \times S^{15} & (p>7).
\end{cases}
\]
Since 
\[
H^*(FII;\Z)=\Z[x_8]/(x_8^3),
\]
we have 
\[
 Q\varphi^*: QH^m(BF_4;\F_p) \to QH^m(B\Spin(9);\F_p)
\]
non-trivial for $m\in\{3,11,15\}$ and $p\ge 5$. Therefore, arguing 
as in the $FI$ case, we obtain homotopy equivalences 
\[
 \Omega(F_4/\Spin(9)) \simeq_p S^7 \times \Omega S^{23} \quad (p\ge 5).
\]

\subsubsection{Type $EIV$} 
It will be convenient to describe the $EIV$ case before that of $EI$. We 
contribute nothing new to this case. By \cite{Harris1}, for odd primes $p$ 
there is a homotopy equivalence 
\[
E_6 \simeq_p E_{6}/F_{4}\times F_{4}.
\] 
So from the decompositions of $E_{6}$ and $F_{4}$ in 
Theorem~\ref{MTexceptional} one obtains homotopy equivalences 
\[
E_{6}/F_4 \simeq \begin{cases}
B(9,17) & (p=5) \\
S^9 \times S^{17} & (p\ge 7). 
\end{cases}
\]

\subsubsection{Type $EI$}
By~\cite{Ishitoya}, for odd primes $p$ there is an isomorphism 
\[ 
H^{\ast}(EI;\F_p)=\F_{p}[e_{8})/(e_{8}^{3})\otimes E(e_{9},e_{17}). 
\] 
Notice that the right side is abstractly isomorphic to 
$H^{\ast}(F_{4}/\Spin(9);\F_{p})\otimes H^{\ast}(E_{6}/F_{4};\F_{4})$. 
Observe that at odd primes, $P\Sp(4) \simeq_p \Spin(9)$  
so $EI=E_{6}/P\Sp(4) \simeq_p E_6/\Spin(9)$. Let 
\[
\phi: E_{6}/F_{4}\to E_{6}
\] 
be the inclusion from the homotopy equivalence 
$E_{6}\simeq_{p} F_{4} \times E_{6}/F_{4}$ and let 
\[
\psi: F_{4}/\Spin(9) \to E_{6}/\Spin(9)
\]
be the map of quotient spaces induced from the factorization of the group homomorphism 
\(\namedright{\Spin(9)}{}{E_{6}}\) 
through $F_{4}$. From the homotopy fibration sequence 
\(\namedddright{E_{6}}{\partial}{E_{6}/\Spin(9)}{}{B\Spin(9)}{}{BE_{6}}\) 
there is a homotopy action 
\[ 
  \theta\colon\namedright{E_{6}\times E_{6}/\Spin(9)}{}{E_{6}/\Spin(9)}  
\] 
which extends $\partial\vee id$. The composition 
\[
 \theta \circ (\phi\times \psi): E_{6}/F_{4} \times F_{4}/\Spin(9) \to E_{6}/\Spin(9),
\] 
therefore induces an isomorphism in mod-$p$ cohomology and so is a 
homotopy equivalence. Combined with the identification of $EIV$ and $FII$ cases,
 we obtain homotopy equivalences 
\[
\Omega E_{6}/P\Sp(4) \simeq \begin{cases}
\Omega B(9,17) \times S^7 \times \Omega S^{23} & (p=5) \\ 
\Omega S^9 \times \Omega S^{17} \times S^7 \times \Omega S^{23}& (p\ge 7). \\
\end{cases}
\]

\subsubsection{Type $EII$}
By Theorem~\ref{MTdecomp} there are homotopy equivalences 
\[
 \SU(2) \cdot \SU(6) \simeq_p \begin{cases}
 S^3 \times B(3,11) \times S^5 \times S^7 \times S^9 & (p=5) \\
 S^3 \times S^3 \times S^5 \times S^7 \times S^9 \times S^{11}& (p>5). 
  \end{cases}
\]
By \cite{Ishitoya2}, for $p\ge 5$
\[
H^*(E_6/\SU(2) \cdot \SU(6); \F_p) = \F_p[x_4, x_6, x_8]/(r_{16}, r_{18}, r_{24})
\] 
for some relations $r_{16}, r_{18}, r_{24}$ in degrees $15, 18, 24$ 
respectively. Thus for $p \ge 5$
\[
 Q\varphi^*: QH^m(E_6; \F_p)\to QH^m(\SU(2) \cdot \SU(6); \F_p),
\]
is non-trivial for $m\in\{3,9,11\}$.
For $p=5$, by Theorem~\ref{qpregAG} there is a homotopy commutative diagram 
\[
\xymatrix{
S^3 \vee A(3,11) \vee S^5 \vee S^7 \vee S^9 \ar[r] \ar[d] &
A(3,11) \vee A(9,17) \vee A(15,23) \ar[d] \\
 \SU(2) \cdot \SU(6) \ar[r]^-{\varphi}  &  E_6. 
}
\]
Proposition \ref{MqID} therefore implies that there is a homotopy equivalence 
\[
 \Omega (E_6/\SU(2)\cdot \SU(6)) \simeq_5
 S^3 \times S^5 \times S^7 \times  
 \Omega S^{17}\times \Omega B(15,23).
 \]
For $p>5$, arguing similarly we obtain 
\[
 \Omega (E_6/\SU(2)\cdot \SU(6)) \simeq_p
 S^3 \times S^5 \times S^7 \times \Omega S^{15} \times 
 \Omega S^{17}\times \Omega S^{23} \quad (p>7).
\]

\subsubsection{Type $EIII$}
By Theorem~\ref{MTdecomp} there are homotopy equivalences 
\[
\Spin(10) \simeq_p \Spin(9) \times S^9 \simeq
\begin{cases}
 S^9 \times B(3,11) \times B(7,15) & (p=5) \\
 S^9 \times B(3,15) \times S^7  \times S^{11} & (p=7). \\
\end{cases}
\]
It is shown in \cite{Iliev-Manivel,Toda-Watanabe} that for $p\ge 5$
\[
H^*(E_6/ T^1\cdot \Spin(10); \F_p)=\F_p[ x_2, x_8]/(r_{18}, r_{24}) 
\] 
for some relations $r_{18}, r_{24}$ in degrees $18,24$. Thus 
\[
Q\varphi^*: QH^m(E_6; \F_p)\to QH^m(T^1\cdot \Spin(10); \F_p)
\]
is non-trivial for $m\in\{3,9,11,15\}$ for $p\ge 5$. Therefore, 
arguing as in the $EII$ case (but modifying slightly to account for the $S^{1}$ 
term by using Remark~\ref{S1variation}) we obtain homotopy equivalences 
\[
 \Omega (E_6/ T^1\cdot \Spin(10)) \simeq_p  
 S^1 \times \Omega S^{17} \times S^7 \times \Omega S^{23}
 \quad (p \ge 5).
\]

\subsubsection{Type $EV$} 
By Theorem~\ref{MTdecomp}, for $p\geq 11$ there are homotopy 
equivalences 
\[
\SU(8)/\{\pm I\} \simeq_p \SU(8) \simeq_p
S^3 \times S^5 \times S^7 \times S^9 \times S^{11} \times S^{13} \times S^{15}.
\] 
By the Appendix, 
\[
 Q\varphi^*:QH^m(E_7; \F_p) \to QH^m(\SU(8)/\{\pm I\}; \F_p)
\]
is non-trivial for $m\in\{3,11,15\}$ when $p\ge 11$. For $p=11$, 
by Theorem~\ref{qpregAG} there is a homotopy commutative diagram 
\[\diagram 
      S^{3}\vee S^{5}\vee S^{7}\vee S^{9}\vee S^{11}\vee S^{13}\vee S^{15}\rto\dto 
          & A(3,23)\vee A(15,35)\vee S^{11}\vee S^{19}\vee S^{27}\dto \\ 
      \SU(8)/\{\pm I\}\rto^-{\varphi} & E_{7}. 
  \enddiagram\] 
Proposition~\ref{MqID} therefore implies that there is a homotopy equivalence 
\[
 \Omega E_7/(\SU(8)/\{\pm I\}) \simeq_5
S^5 \times S^7 \times S^9 \times S^{13} \times
\Omega S^{19}\times \Omega S^{23} \times \Omega S^{27} \times \Omega S^{35}. 
\] 
For $p>11$, arguing similarly we obtain 
\[
 \Omega E_7/(\SU(8)/\{\pm I\}) \simeq_p 
S^5 \times S^7 \times S^9 \times S^{13} \times
\Omega S^{19}\times \Omega S^{23} \times \Omega S^{27} \times \Omega S^{35}
\quad (p\ge 11). 
\]

\subsubsection{Type $EVI$}
By Theorem~\ref{MTdecomp}, there are homotopy equivalences 
 \[
 \Spin(12) \simeq_p S^3 \times S^7  \times S^{11} \times S^{11} \times S^{15} \times S^{19} \quad (p\ge 11).
 \]  
By~\cite{Nakagawa2}, for $p\ge 5$
\[
H^*(E_7/T^1\cdot \Spin(12); \F_p) = 
\F_p[x_2, x_{8}, x_{12}]/(r_{24}, r_{28}, r_{36})
\] 
for some relations $r_{24},r_{28},r_{36}$ in degrees $24,28,36$ respectively. 
From the fibre sequence
\[
 S^2 \hookrightarrow E_7/T^1\cdot \Spin(12) \to E_7/\SU(2)\cdot \Spin(12)
\]
we therefore obtain 
\[
H^*(E_7/\SU(2) \cdot \Spin(12); \F_p) = \F_p[x_4, x_{8}, x_{12}]/I
\] 
for some ideal $I$ consisting of elements of degrees~$\geq 24$. Hence 
\[
 Q\varphi^*: QH^m(E_7; \F_p) \to QH^m(\SU(2)\cdot \Spin(12); \F_p)
\]
is non-trivial for $m\in\{3,11,15,19\}$ when $p\ge 11$. Therefore, 
arguing as in the $EV$ case we obtain homotopy equivalences 
\[
\Omega E_7/\SU(2)\cdot \Spin(12) \simeq_p S^3 \times
S^7 \times S^{11} \times \Omega S^{23} \times \Omega S^{27} \times \Omega S^{35} \quad (p\ge 11).
\]

\subsubsection{Type $EVII$} 
By Theorem~\ref{MTexceptional} there are homotopy equivalences 
\[
E_{6} \simeq_p  
\begin{cases}
 B(3,23) \times S^{9}\times S^{11}\times S^{15}\times S^{17} & (p=11) \\
 S^{3}\times S^{9}\times S^{11}\times S^{15}\times S^{17}\times S^{23} & (p>11). \\
\end{cases}
\]
By \cite{Chaput-Manivel,Watanabe}, for $p\ge 11$
\[
H^*(E_7/T^1\cdot E_6; \F_p) = \F_p[x_2, x_{10}, x_{18}]/(r_{20}, r_{28}, r_{36})
\] 
for some relations $r_{20},r_{28},r_{36}$ in degrees $20, 28, 36$ respectively. Thus 
\[
Q\varphi^*: QH^m(E_7; \F_p) \to QH^m(T^1\cdot E_6; \F_p)
\]
is non-trivial for $m\in\{3,11,15,23\}$ when $p\ge 11$. Therefore, arguing as in 
the $EV$ case (modifying slightly to account for the $S^{1}$ term by 
using Remark~\ref{S1variation}) we obtain homotopy equivalences 
\[
 \Omega E_7/T^1\cdot E_6 \simeq_p S^1\times S^9 \times S^{17} \times \Omega S^{19} \times \Omega S^{27} \times \Omega S^{35} \quad (p \ge 11).
\]

\subsubsection{Type $EVIII$}
Using Theorem~\ref{MTdecomp}, there are homotopy equivalences 
\[
 \Ss(16) \simeq_p S^{15}\times \Sp(7) \simeq_p \begin{cases}
 B(3,23) \times B(7,27) \times S^{11} \times S^{15} \times S^{15} \times S^{19} & (p=11) \\ 
 B(3,27) \times S^7 \times S^{11} \times S^{15} \times S^{15} \times S^{19} \times S^{23} & (p=13) \\ 
 S^3 \times S^7 \times S^{11} \times S^{15} \times S^{15} \times S^{19} \times S^{23} \times S^{27} & (p\ge 17). \\ 
 \end{cases}
\] 
By \cite{Hamanaka-Kono} and \cite{Kaji-Kishimoto}, 
\[
 Q\varphi^*: QH^*(E_8;\F_p) \to QH^*(\Ss(16);\F_p)
\]
is non-trivial for $m\in\{3,15,23,27\}$ when $p>5$. For $p=11$, 
by Theorem~\ref{qpregAG} there is a homotopy commutative diagram 
\[\diagram 
       A(3,23)\vee A(7,27)\vee S^{11}\vee S^{15}\vee S^{15}\vee S^{19}\rto\dto 
          & A(3,23)\vee A(15,35)\vee A(27,47)\vee A(39,59)\dto \\ 
       \Ss(16)\rto^-{\varphi} & E_{8}. 
  \enddiagram\] 
Proposition~\ref{MqID} therefore implies that there is a homotopy equivalence 
\[
\Omega E_8/\Ss(16) \simeq_{11} S^7 \times S^{11} \times S^{15} \times S^{19} 
     \times \Omega S^{35} \times \Omega S^{47} \times \Omega B(39,59).\] 
For $p>11$, arguing similarly we obtain 
\[
\Omega E_8/\Ss(16) \simeq_p \begin{cases}
S^7 \times S^{11} \times S^{15} \times S^{19} \times \Omega S^{39} \times 
    \Omega S^{47} \times \Omega B(35,59) & (p=13) \\
S^7 \times S^{11} \times S^{15} \times S^{19} \times \Omega S^{35} \times 
   \Omega S^{39}  \times \Omega S^{47} \times \Omega S^{59} & (p\ge 17). \\
\end{cases}
\]

\subsubsection{Type $EIX$} 
Recall the four cases for the homotopy decomposition of $E_{7}$ in 
Theorem~\ref{MTexceptional} when $p\geq 11$. By~ \cite{Nakagawa1},
\[
 H^*(E_8/T^1\cdot E_7; \F_p)=
 \F_p[x_2, x_{12}, x_{20}]/(r_{40}, r_{48}, r_{60}),
\] 
for some relations $r_{40},r_{48},r_{60}$ in degrees $40,48,60$ respectively. 
From the fibre sequence
\[
 S^2 \hookrightarrow E_8/T^1\cdot E_7 \to E_8/\SU(2)\cdot E_7
\]
we obtain 
\[
 H^*(E_8/\SU(2)\cdot E_7; \F_p) = \F_p [x_4, x_{12}, x_{20}]/I
\] 
where $I$ is some ideal consising of elements in degrees~$\geq 40$. Thus 
\[
 Q\varphi^m: QH^*(E_8; \F_p) \to QH^m(\SU(2)\cdot E_7; \F_p)
\]
is non-trivial for $m\in\{3,15,23,27,35\}$ when $p\ge 11$.
Arguing similarly to the $EVIII$ case we obtain homotopy equivalences 
\[
\Omega E_8/\SU(2)\cdot E_7 \simeq_p \begin{cases}
S^3 \times S^{11} \times S^{19} \times \Omega S^{47} \times \Omega B(39,59) & (p=11) \\
S^3 \times S^{11} \times S^{19} \times \Omega S^{39} \times \Omega S^{47} \times \Omega S^{59} & (p\ge 13). 
\end{cases}
\]
\medskip 

Summarising the results for the exceptional cases, we have the following 
(together with exponent information which will be proved later in 
Section~\ref{sec:exponents}). 

\begin{Thm}\label{thm:exceptional} 
For $p$ an odd prime, there are homotopy equivalences: \\ 
{\footnotesize
\begin{tabular}{|l|l||l|l|} \hline
Type & \hspace{1cm} $G/H$ 
    & \hspace{2.5cm} Homotopy type of $\Omega(G/H)$ 
    & Exponent\\  
\hline \hline 
$G$ & 
$G_{2}/\SO(4)$ &
$\begin{array}{ll}
S^{3}\times\Omega S^{11} & p\geq 5 
\end{array}$  
& $\begin{array}{l} 
   =p^{5} 
\end{array}$  
\\ \hline 
$FI$ & 
$F_{4}/\SU(2)\cdot \Sp(3)$ &
$\begin{array}{ll}
S^{3}\times S^{7}\times\Omega B(15,23) & p=5 \\ 
S^{3}\times S^{7}\times\Omega S^{15}\times\Omega S^{23} & p\geq 7 
\end{array}$ 
& $\begin{array}{l} 
\leq 5^{12} \\ =p^{11} 
\end{array}$ 
\\ \hline 
$FII$ & 
$F_{4}/\Spin(9)$ &
$\begin{array}{ll}
S^{7}\times\Omega S^{23} & p\geq 5 
\end{array}$ 
& $\begin{array}{l} 
   =p^{11} 
\end{array}$ 
\\ \hline 
$EI$ & 
$E_{6}/P\Sp(4)$ & 
$\begin{array}{ll} 
S^{7}\times\Omega B(9,17)\times\Omega S^{23} & p=5 \\ 
S^{7}\times\Omega S^{9}\times\Omega S^{17}\times\Omega S^{23} & p\geq 7 
\end{array}$ 
& $\begin{array}{l} 
= 5^{11} \\ =p^{11} 
\end{array}$ 
\\ \hline 
$EII$ & 
$E_{6}/\SU(2)\cdot \SU(6)$ & 
$\begin{array}{ll}
S^3 \times S^5 \times S^7 \times  
 \Omega S^{17}\times \Omega B(15,23) & p=5 \\ 
S^3 \times S^5 \times S^7 \times  
 \Omega^{15} \times \Omega S^{17}\times \Omega S^{23} & p\geq 7 
\end{array}$ 
& $\begin{array}{l} 
\leq 5^{12} \\ = p^{11} 
\end{array}$ 
\\ \hline 
$EIII$ & 
$E_{6}/T^{1}\cdot \Spin(10)$ & 
$\begin{array}{ll}
S^{1}\times S^{7}\times\Omega S^{17}\times\Omega S^{23} & p\geq 5 
\end{array}$ 
& $\begin{array}{l} 
   =p^{11} 
\end{array}$ 
\\ \hline 
$EIV$ & 
$E_{6}/F_{4}$ &
$\begin{array}{ll} 
\Omega B(9,17) & p=5 \\ 
\Omega S^{9}\times\Omega S^{17} & p\geq 7 
\end{array}$ 
& $\begin{array}{l} 
\leq 5^{9} \\ =p^{8} 
\end{array}$ 
\\ \hline 
$EV$ & 
$E_{7}/(\SU(8)/\{ \pm I\})$ &
$\begin{array}{ll}
S^{5}\times S^{7}\times S^{9}\times S^{13}\times\Omega S^{19} 
    \times\Omega S^{23}\times\Omega S^{27}\times\Omega S^{35} & p\geq 11 
\end{array}$ 
& $\begin{array}{l} 
   =p^{17} 
\end{array}$ 
\\ \hline 
$EVI$ & 
$E_{7}/\SU(2)\cdot \Spin(12)$ & 
$\begin{array}{ll}
S^{3}\times S^{7}\times S^{11}\times\Omega S^{23}\times\Omega S^{27} 
    \times\Omega S^{35} & p\geq 11  
\end{array}$ 
& $\begin{array}{l} 
   =p^{17} 
\end{array}$ 
\\ \hline 
$EVII$ & 
$E_{7}/T^{1}\cdot E_{6}$ & 
$\begin{array}{ll}
S^{1}\times S^{9}\times S^{17}\times\Omega S^{19}\times\Omega S^{27} 
    \times\Omega S^{35} & p\geq 11 
\end{array}$ 
& $\begin{array}{l} 
   =p^{17} 
\end{array}$ 
\\ \hline 
$EVIII$ & 
$E_{8}/\Ss(16)$ &
$\begin{array}{ll} 
S^{7}\times S^{11}\times S^{15}\times S^{19}\times\Omega S^{35} 
    \times\Omega B(39,59)\times\Omega S^{47} & p=11 \\ 
S^{7}\times S^{11}\times S^{15}\times S^{19}\times\Omega B(35,59) 
    \times\Omega S^{39}\times\Omega S^{47} & p=13 \\ 
S^{7}\times S^{11}\times S^{15}\times S^{19}\times\Omega S^{35} 
     \times\Omega S^{39}\times\Omega S^{47}\times\Omega S^{59} & p\geq 17 
\end{array}$ 
& $\begin{array}{l} 
\leq 11^{30} \\ \leq 13^{30} \\ =p^{29} 
\end{array}$ 
\\ \hline 
$EIX$ & 
$E_{8}/\SU(2)\cdot E_{7}$ & 
$\begin{array}{ll}
S^{3}\times S^{11}\times S^{19}\times\Omega B(39,59)\times\Omega S^{47} 
     & p=11 \\ 
S^{3}\times S^{11}\times S^{19}\times\Omega S^{39}\times\Omega S^{47} 
     \times\Omega S^{59} & p\geq 13 
\end{array}$  
& $\begin{array}{l} 
\leq 11^{30} \\ =p^{29} 
\end{array}$ 
\\ \hline
\end{tabular}
} 

$\qqed$ 
\end{Thm}

\begin{Rem} 
\label{Terzicremark2} 
Two of the decompositions in the previous table deloop. Harris~\cite{Harris1} 
showed that $E_{6}/F_{4}\simeq_{5} B(9,17)$ and 
$E_{6}/F_{4}\simeq_{p} S^{9}\times S^{17}$ for $p\geq 7$, and in this 
paper we show that $E_{6}/P\Sp(4)\simeq_{p}E_{6}/F_{4}\times F_{4}/\Spin(9)$ 
for $p\geq 3$. 
\end{Rem}

\begin{Rem}\label{rem:terzic-exceptional}
Terzic's computation of the rational homotopy groups
(\cite{Terzic})
can be easily reproduced from these decompositions.
We found minor mistakes in her calculations for 
$G_2/\SO(4)$ and $E_6/\SU(2)\cdot \SU(6)$.
See also Remark \ref{rem:terzic-exceptional} for classical cases.
\end{Rem}

\section{Limitations and extensions of the methods} 

In this section we examine the boundaries of our methods and results. 
It is natural to ask whether the loop space decompositions of symmetric 
spaces deloop, and whether the methods can be extended to apply in 
cases that are not quasi-$p$-regular. 
\medskip 

\subsection{Impossibility of delooping}
We gave decompositions for the loop spaces of symmetric spaces.
It is reasonable to ask whether they actually come from 
decompositions of symmetric spaces themselves. Kumpel~\cite{Kumpel} and 
Mimura~\cite{Mimura2} showed that if the homotopy fibration 
\(\nameddright{H}{}{G}{}{G/H}\) 
is totally non-cohomologous to zero then the symmetric space 
will decompose, delooping our results. This holds for $SU(2n+1)/\SO(2n+1)$, 
$\SU(2n)/\Sp(n)$, \mbox{$\Spin(2n)/\Spin(2n-1)$} and $E_{6}/F_{4}$. However, 
in general a delooping does not exist, as we now see with the particular example 
of $FI=F_{4}/\SU(2)\cdot \Sp(3)$.

We have shown that
\[
 \Omega FI \simeq_5 S^3 \times S^7 \times \Omega B(15,23).
\]
However, this decomposition does not deloop, as we now show. The 
following calculation will be needed.  

\begin{Thm}[\cite{Ishitoya-Toda}]
\[
H^{*}(FI;\F_{p}) = 
\dfrac{\F_{p}[f_{4},f_{8}]}{(f_{4}^{3}-12f_{4}f_{8}+8f_{12},
f_{4}f_{12}-3f_{8}^{2}, f_{8}^{3}-f_{12}^{2})}.
\] 
$\qqed$ 
\end{Thm} 

In particular, 
\[
H^{*}(FI;\F_{5}) = 
\dfrac{\F_{5}[f_{4},f_{8},f_{12}]}{(f_{4}^{3}-2f_{4}f_{8}-2f_{12},
f_{4}f_{12}-3f_{8}^{2}, f_{8}^{3}-f_{12}^{2})}.
\]
We will show that this ring cannot be a non-trivial tensor product of two rings.
From the relations we obtain 
\begin{align*}
3(f_{4}^{3}-2f_{4}f_{8}-2f_{12}) & \Rightarrow f_{12}=3f_{4}^{3}-f_{4}f_{8} \\
f_{4}f_{12}-3f_{8}^{2} & \Rightarrow f_{4}^{4}-2f_{4}^{2}f_{8}-f_{8}^{2} \\
f_{8}^{3}-f_{12}^{2}.
\end{align*}
If a splitting exists, there should be a substitution 
\[
 f_{4} \mapsto f_{4}, f_{8} \mapsto a f'_{8} + b f_{4}^{2}, \quad a \in \F^{\times}_{5}, b\in \F_{5}
\]
such that the relation
$f_{4}^{4}-2f_{4}^{2}f_{8}-f_{8}^{2}$ lies in $\F_{5}[f_{4}] \cup \F_{5}[f'_{8}]$.
However, this is impossible.
Therefore there is no non-trivial product decomposition for $FI$
localised at $p=5$.

\subsection{Non-quasi-$p$-regular cases} 

We study examples of Lie group homomorphisms 
\(\namedright{H}{\varphi}{G}\) 
when $H$ and/or $G$ are not quasi-$p$-regular. In the first three 
examples, the methods from Sections 2 to 4 hold and a homotopy decomposition 
of $\Omega(G/H)$ is obtained, while in the final two examples potential 
obstructions appear. 

All the examples occur at the prime $p=7$, and relate to the homotopy equivalences 
\[E_{7}\simeq_{7} B(3,15,27)\times B(11,23,35)\times S^{19}\] 
\[E_{8}\simeq_{7} B(3,15,27,39)\times B(23,35,47,59)\] 
established in~\cite{MNT}. 
\medskip  

\noindent 
1. $\mathbf{EV=E_{7}/(\SU(8)/\{\pm I\})}$. Here, 
$\SU(8)/\{\pm I\}\simeq_{7}\SU(8)\simeq_{7} 
     B(3,15)\times S^{5}\times S^{7}\times S^{9}\times S^{11}\times S^{13}$. 
We hope to apply Theorem~\ref{decomp}. Consider the composite 
\[\phi\colon\namedddright{A(3,15)\vee S^{5}\vee S^{7}\vee S^{9}\vee S^{11}\vee S^{13}} 
      {}{\SU(8)}{\varphi}{E_{7}}{\simeq}{B(3,15,27)\times B(11,23,35)\times S^{19}}.\] 
By~\cite{MNT}, the homotopy groups of  
$B(3,15,27)\times B(11,23,35)\times S^{19}$ 
are zero in dimensions $\{5,7,9,13\}$, so $\phi$ factors through a map 
\(\phi^{\prime}\colon\namedright{A(3,15)\vee S^{11}}{} 
      {B(3,15,27)\times B(11,23,35)\times S^{19}}\). 
As well, by~\cite{MNT} $\pi_{t}(B(11,23,35))=0$ for $t\in\{3,15\}$, 
$\pi_{11}(B(3,15,27))=0$ and $\pi_{t}(S^{19})=0$ for $t\in\{3,11,15\}$, 
so the map $\phi^{\prime}$ is determined by the maps 
\(\phi_{1}^{\prime}\colon\namedright{A(3,15)}{}{B(3,15,27)}\) 
and 
\(\phi_{2}^{\prime}\colon\namedright{S^{11}}{}{B(11,23,35)}\). 
The $15$-skeleton of $B(3,15,27)$ is $A(3,15)$ so $\phi_{1}^{\prime}$ 
factors as a composite 
\(\nameddright{A(3,15)}{g_{1}}{A(3,15,27)}{}{B(3,15,27)}\) 
for some map $g_{1}$. Similarly, $\phi_{2}^{\prime}$ factors as a composite 
\(\nameddright{S^{11}}{g_{2}}{A(11,23,35)}{}{B(11,23,35)}\) 
for some map $g_{2}$. Hence there is a homotopy commutative diagram 
\[\diagram 
         A(3,15)\vee S^{5}\vee S^{7}\vee S^{9}\vee S^{11}\vee S^{13}\rto^-{Q}\dto 
             & A(3,15)\vee S^{11}\rto^-{g_{1}\vee g_{2}} 
             & A(3,15,27)\vee A(11,23,35)\vee S^{19}\dto \\ 
         \SU(8)\rrto^-{\varphi} & &  E_{7} 
  \enddiagram\] 
where $Q$ is the pinch map. Therefore, noting that $M(S^{2n+1})\simeq S^{2n+1}$, 
by Theorem~\ref{decomp} and Corollary~\ref{decompcor}, the homotopy fibre 
of the map 
\(\namedright{\SU(8)}{\varphi}{E_{7}}\) 
is homotopy equivalent to the homotopy fibre of the composite 
\begin{align*} 
    M(A(3,15))\times S^{5}\times S^{7}\times S^{9}\times S^{11}\times S^{13}  
    & \stackrel{\pi}{\lllarrow} M(A(3,15))\times S^{11} \\ 
    & \stackrel{M(g_{1})\times M(g_{2})}{\lllarrow} 
       M(A(3,15,27))\times M(A(11,23,35))\times S^{19} 
\end{align*} 
where $\pi$ is the projection. 

In the Appendix it is shown that 
\[Q\varphi^{\ast}\colon\namedright{QH^{m}(E_{7})}{}{QH^{m}(\SU(8)/\{\pm I\})}\] 
is nontrivial for $m\in\{3,11,15\}$. Thus $g_{1}^{\ast}$ and $g_{2}^{\ast}$ are 
onto in mod-$7$ cohomology, implying that $M(g_{1})^{\ast}$ and $M(g_{2})^{\ast}$ 
are onto in mod-$7$ cohomology. Therefore, arguing as in Proposition~\ref{MqID}, 
there is a homotopy equivalence 
\[\Omega(E_{7}/(\SU(8)/\{\pm I\})\simeq_{7} S^{5}\times S^{7}\times S^{9}\times S^{13} 
       \times\Omega S^{27}\times\Omega B(23,35)\times\Omega S^{19}.\] 
\medskip 

\noindent 
2. $\mathbf{EVI=E_{7}/\SU(2)\cdot\Spin(12)}$. Here, 
$\SU(2)\cdot\Spin(12)\simeq_{7} \SU(2)\times\Spin(12)\simeq_{7} 
       S^{3}\times B(3,15)\times B(7,19)\times S^{11}\times S^{11}$. 
Arguing as in the previous case, we obtain maps 
\(g_{1}\colon\namedright{S^{3}\vee A(3,15)}{}{A(3,15,27)}\), 
\(g_{2}\colon\namedright{S^{11}\vee S^{11}}{}{A(11,23,35)}\) 
and 
\(g_{3}\colon\namedright{A(7,19)}{}{S^{19}}\) 
and a homotopy commutative diagram 
\[\diagram 
         (S^{3}\vee A(3,15))\vee (S^{11}\vee S^{11})\vee A(7,19) 
                  \rrto^-{g_{1}\vee g_{2}\vee g_{3}}\dto 
             & & A(3,15,27)\vee A(11,23,35)\vee S^{19}\dto \\ 
         S^{3}\times\Spin(12)\rrto^-{\varphi} & &  E_{7} 
  \enddiagram\] 
As in Section 5.2.8, $Q\varphi^{\ast}$ is nonzero in degrees 
$\{3,11,15,19\}$, so arguing as in the previous case we obtain a homotopy 
equivalence 
\[\Omega(E_{7}/\SU(2)\cdot\Spin(12))\simeq_{7} S^{3}\times\Omega S^{27}\times 
      S^{11}\times\Omega B(23,35)\times S^{7}.\] 
\medskip 

\noindent 
3. $\mathbf{EVII=E_{7}/T^{1}\cdot E_{6}}$. 
Here, $T^{1}\cdot E_{6}\simeq_{7} T^{1}\times E_{6}\simeq_{7} 
      S^{1}\times B(3,15)\times B(11,23)\times S^{9}\times S^{17}$. 
Arguing as in the first case, we obtain maps 
\(g_{1}\colon\namedright{A(3,15)}{}{A(3,15,27)}\) 
and 
\(g_{2}\colon\namedright{A(11,23)}{}{A(11,23,35)}\), 
and a homotopy commutative diagram 
\[\small\diagram 
         S^{1}\vee A(3,15)\vee A(11,23)\vee S^{9}\vee S^{17}\rto^-{Q}\dto 
             & A(3,15)\vee A(11,23)\rto^-{g_{1}\vee g_{2}} 
             & A(3,15,27)\vee A(11,23,35)\vee S^{19}\dto \\ 
         S^{1}\times E_{6}\rrto^-{\varphi} & &  E_{7} 
  \enddiagram\] 
where $Q$ is the pinch map. As in Section 5.2.9, $Q\varphi^{\ast}$ is nonzero 
in degrees $\{3,11,15,23\}$, so arguing as in the first case we obtain a 
homotopy equivalence 
\[\Omega(E_{7}/T^{1}\cdot E_{6})\simeq_{7} S^{1}\times\Omega S^{27}\times 
       \Omega S^{35}\times S^{9}\times S^{17}\times\Omega S^{19}.\] 
\medskip 

\noindent 
4. $\mathbf{EVIII=E_{8}/\Ss(16)}$. 
Here, 
$\Ss(16)\simeq_{7}\simeq\Spin(16)\simeq_{7} B(3,15,27)\times 
     B(7,19)\times B(11,23)\times S^{15}$. 
We hope to apply Theorem~\ref{decomp}. Consider the composite 
\[\phi\colon\namedright{A(3,15,27)\vee A(7,19)\vee A(11,23)\vee S^{15}} 
      {}{\Spin(16)\hspace{2cm}}\] 
\[\hspace{6cm}\nameddright{}{\varphi}{E_{8}}{\simeq}{B(3,15,27,39)\times B(23,35,47,59)}.\] 
By~\cite{MNT}, the homotopy groups of $B(3,15,27,39)\times B(23,35,47,59)$ 
are zero in dimensions $\{7,19\}$ so $\phi$ factors through a map 
\(\phi^{\prime}\colon\namedright{A(3,15,27)\vee A(11,23)\vee S^{15})} 
       {}{B(3,15,27,39)\times B(23,35,47,59)}\). 
By~\cite{MNT}, $\pi_{t}(B(23,35,47,59)=0$ for $t\in\{3,15,27\}$ and 
$\pi_{t}(B(3,15,27,35)=0$ for $t\in\{11,23\}$, so the map $\phi^{\prime}$ 
is determined by maps 
\(\phi_{1}^{\prime}\colon\namedright{A(3,15,27)\vee S^{15}}{}{B(3,15,27,39)}\) 
and 
\(\phi_{2}^{\prime}\colon\namedright{A(11,23)}{}{B(23,35,47,59)}\). 
Notice that the $27$-skeleton of $B(3,15,27,39)$ is $A(3,15,27)\cup e^{18}$, 
and $\pi_{27}(S^{18})\cong\mathbb{Z}/7\mathbb{Z}$. Thus there is a potential 
obstruction to lifting $\phi_{1}^{\prime}$ to a map 
\(\namedright{A(3,15,27)\vee S^{15}}{}{A(3,15,27,39)}\). 
It is unclear whether the obstruction vanishes. If not, then Theorem~\ref{decomp} 
cannot be applied and the homotopy type of $\Omega(E_{8}/\Ss(16))$ at $p=7$ 
would remain undetermined. 
\medskip 

\noindent 
5. $\mathbf{EVIX=E_{8}/SU(2)\cdot E_{7}}$. As in the previous example, we 
obtain an obstruction to lifting 
\(\phi_{1}^{\prime}\colon\namedright{S^{3}\vee A(3,15,27)}{}{B(3,15,27,39)}\) 
to $A(3,15,27,39)$, which leaves unresolved the homotopy type of 
$\Omega(E_{8}/SU(2)\cdot E_{7})$ at $p=7$. 

\begin{Rem} 
An important difference between the three $E_{7}$ examples that worked and 
the two~$E_{8}$ examples that did not is that the domain in the three $E_{7}$ 
examples were all quasi-$p$-regular while this was not the case in 
the $E_{8}$ examples. 
\end{Rem}

\section{Exponents} 
\label{sec:exponents} 

Recall that, for a prime $p$, the \emph{$p$-primary homotopy exponent} 
of a space $X$ is the least power of $p$ that annihilates the $p$-torsion 
in $\pi_{\ast}(X)$. If the $p$-primary exponent is $p^{r}$, write 
$\exp_{p}(X)=p^{r}$. The homotopy decompositions of $\Omega(G/H)$ 
allow us to find precise exponents or upper and lower bounds on the 
exponent of $G/H$. 

Observe that in every homotopy decomposition of $\Omega(G/H)$ in 
Theorems~\ref{thm:classical} and~\ref{thm:exceptional}, the factors 
are either spheres, sphere bundles 
over spheres, or the loops on either of these two. Exponent information 
about these spaces is known. A precise exponent for spheres was 
determined in~\cite{CMN2}, and exponent bounds for spaces of the 
form $B(2m-1,2m+2p-3)$ was determined in~\cite{Davis-Theriault}. 

\begin{Thm}[\cite{CMN2}] 
   \label{CMNexp} 
   Let $p\geq 5$. Then $\exp_{p}(S^{2n+1})=p^{n}$.~$\qqed$ 
\end{Thm} 

\begin{Thm}[\cite{Davis-Theriault}]
   \label{DT} 
   Let $p\geq 5$. Then $\exp_{p}(B(3,2p+1))=p^{p+1}$ and 
   for $m>2$, 
   \[p^{m+p-2}\leq\exp_{p}(B(2m-1,2m+2p-3))\leq p^{m+p-1}.\] 
   $\qqed$ 
\end{Thm} 

Suppose that $X$ is a product of spheres and spaces $B(2i-1,2i+2p-3)$ 
for various~$i$. Rationally, $X$ is homotopy equivalent to a product of odd 
dimensional spheres, say $X\simeq_{\mathbb{Q}}\prod_{i=1}^{\ell} S^{2m_{i}+1}$. 
The \emph{type} of $X$ is the list $\{m_{1},\ldots,m_{\ell}\}$ where - 
relabelling if necessary - we may assume that $m_{1}\leq\cdots\leq m_{\ell}$. 
Theorems~\ref{CMNexp} and~\ref{DT} immediately imply that the exponent 
of $X$ depends only on the exponent of the factors of $X$ containing a 
generator in cohomology of degree~\mbox{$2m_{\ell}+1$}. Explicitly, 
$\exp_{p}(X)=p^{m_{\ell}}$ if each factor of $X$ containing a generator 
in cohomology of degree~$2m_{\ell}+1$ is a sphere, and 
$p^{m_{\ell}} \leq \exp_{p}(X)\leq p^{m_{\ell}+1}$ if at least one factor of $X$ 
containing a generator in cohomology of degree~$2m_{\ell}+1$ is 
$B(2m_{\ell}-2p+3, 2m_{\ell}+1)$. In our case, observe that the homotopy 
decompostions for $\Omega(G/H)$ in the classical cases listed in  
Theorem~\ref{thm:classical} imply that the factor containing a generator 
in cohomology of maximal degree is of the form $B(2i-1,2i+2p-3)$ 
only for $\SU(2n+1)/\SO(2n+1)$, $\SU(2n)/\SO(2n)$ and $\SU(2n)/\Sp(n)$ 
when $n=p-1$. Thus we have the following. 

\begin{Thm} 
   \label{expbounds} 
   For $p\geq 5$, there are exponent bounds: 
\[
\begin{array}{|c|c||c|c|}
\hline
\text{Type} & G/H & p\geq 5 & \text{Exponent} \\
\hline
AI & \SU(2n+1)/\SO(2n+1) & p>n &
\left\{\begin{array}{ll} 
    \leq p^{4n+2}\ \ \mbox{if $p-1=n$} \\ 
    =p^{4n+1}\ \ \mbox{if $p-1>n$}  
\end{array}\right. 
\\ 
 & \SU(2n)/\SO(2n) & p>n &
\left\{\begin{array}{ll} 
    \leq p^{4n}\ \ \mbox{if $p-1=n$}  \\ 
    =p^{4n-1}\ \ \mbox{if $p-1>n$} 
\end{array}\right. 
\\ \hline
AII & \SU(2n)/\Sp(n) & p>n &
\left\{\begin{array}{ll} 
    \leq p^{4n}\ \ \mbox{if $p-1=n$} \\ 
    =p^{4n-1}\ \ \mbox{if $p-1>n$} 
\end{array}\right. 
\\ \hline
AIII & \dfrac{U(n)}{U(m)\times U(n-m)}^\dag & p>n/2 & =p^{2n-1}  
\\ \hline
BDI & \dfrac{\SO(2n+1)}{\SO(2m)\times \SO(2(n-m)+1)}^\dag & p>n & =p^{4n-1} 
\\
 & \dfrac{\SO(2n+1)}{\SO(2m-1)\times \SO(2(n-m)+2)}^\ddag & p>n & =p^{4n-1} 
\\
 & \dfrac{\SO(2n+2)}{\SO(2m+1)\times \SO(2(n-m)+1)}^\dag & p>n & =p^{4n-1} 
\\
 & \dfrac{\SO(2n+2)}{\SO(2m)\times \SO(2(n-m)+2)}^\ddag & p>n-1 & =p^{4n-1} 
\\ \hline
CI & \Sp(n)/U(n) & p>n & =p^{4n-1} 
\\ \hline
CII & \dfrac{\Sp(n)}{\Sp(m)\times \Sp(n-m)}^\dag & p>n & =p^{4n-1} 
\\ \hline
DIII & \SO(2n)/U(n) & p>n-1 & =p^{4n-3} 
\\ \hline
\end{array}
\] 
\begin{itemize}
\item for $\dag$, we assume $2m\le n$
\item for $\ddag$, we assume $2m\le n+1$
\end{itemize} 
$\qqed$ 
\end{Thm}

Theorems~\ref{CMNexp} and~\ref{DT} also imply the exponent bounds 
listed in Theorem~\ref{thm:exceptional}.

\section*{Appendix}
\newcommand{\perpendicular}{\alpha^{\bot}\!}
\newcommand{\newvar}[1]{t_{#1}}
\newcommand{\reflection}{\rho}
\newcommand{\sroot}{\tau}
\newcommand{\generator}[1]{x_{#1}}

For $p>5$, we show that 
\[
 Qi^*:QH^m(E_7;\F_p) \to QH^m(\SU(8)/C;\F_p)
\]
is non-trivial for $m\in\{3,11,15\}$, 
where $C=\{\pm I\}$. 
To see this we show that 
\[
 Qi^*:QH^m(BE_7;\F_p) \to QH^m(B(\SU(8)/C;\F_p))
\]
is non-trivial for $m\in\{4,12,16\}$ via the Weyl group invariant subrings.

The extended Dynkin-Coxeter diagram for $E_7$ is as follows:
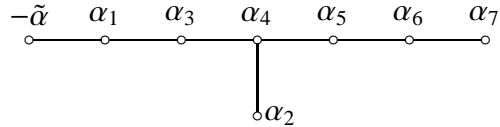
\begin{center}
\newcommand{\simpleroot}[2]{\put(#1){\circle{#2}}}
\newcommand{\simplerootname}[3]{\put(#1){\makebox(0,0)[#2]{$#3$}}}
\newcommand{\angleone}[3]{\put(#1){\line(#2){#3}}}
\newcommand{\anglemulti}[5]{\multiput(#1)(#2){#3}{\line(#4){#5}}}
\unitlength 1mm
\begin{picture}(47,20)
\simpleroot{1,11}{1}
\simpleroot{11,11}{1}
\simpleroot{21,11}{1}
\simpleroot{31,11}{1}
\simpleroot{41,11}{1}
\simpleroot{51,11}{1}
\simpleroot{61,11}{1}
\simpleroot{31,1}{1}
\angleone{1.5,11}{1,0}{9}
\angleone{11.5,11}{1,0}{9}
\angleone{21.5,11}{1,0}{9}
\angleone{31.5,11}{1,0}{9}
\angleone{41.5,11}{1,0}{9}
\angleone{51.5,11}{1,0}{9}
\angleone{31,1.5}{0,1}{9}
\simplerootname{1,13}{b}{-\tilde\alpha}
\simplerootname{10.9,13}{b}{\alpha_1}
\simplerootname{32,1}{l}{\alpha_2}
\simplerootname{20.9,13}{b}{\alpha_3}
\simplerootname{30.9,13}{b}{\alpha_4}
\simplerootname{40.9,13}{b}{\alpha_5}
\simplerootname{50.9,13}{b}{\alpha_6}
\simplerootname{60.9,13}{b}{\alpha_7}
\end{picture}
\end{center}
We adopt a basis $t_i$'s satisfying
\[
\tilde\alpha=\newvar{1}-\newvar{2},\  
\alpha_1=\newvar{3}-\newvar{2},\ 
\alpha_2=\frac{(\newvar{1}+\cdots+\newvar{4})-(\newvar{5}+\cdots+\newvar{8})}{2},\ 
\alpha_i=\newvar{i+1}-\newvar{i}\ (3\leq i\leq7).
\]
The Weyl group $W(A_7)$ for $\SU(8)/C$ is generated by 
the reflection corresponding to $\alpha_i$ ($i\neq2$) and $\tilde{\alpha}$, and
\[
H^*(B(\SU(8)/C);\F_p)=H^*(BT;\F_p)^{W(A_7)}=\F_p[c_2,\ldots,c_8],
\]
where $c_i$ is the $i$-th elementary symmetric polynomial in $\newvar{j}$'s.
Let $\reflection$ be the reflection corresponding to $\alpha_2$.  
We check that algebra generators in degrees 4, 12, and 16 contain
$c_{2},c_{6}$ and $c_{8}$, respectively, in 
\[
H^*(BE_7;\F_p)=H^*(BT;\F_p)^{W(E_7)}=\F_p[c_2,\ldots,c_8]^{\reflection}.
\]

Let $a_i$ and $b_i$ be the $i$-th elementary symmetric polynomials 
 in $\newvar{1},\ldots,\newvar{4}$ and $\newvar{5},\ldots,\newvar{8}$, respectively.
Notice that $\alpha_2=
\frac{\newvar{1}+\newvar{2}+\newvar{3}+\newvar{4}-
\newvar{5}-\newvar{6}-\newvar{7}-\newvar{8}}{2}=a_1$ and
$c_i=\sum_{j+k=i}a_jb_k$.

Denote $\frac{\alpha_2}{2}$ by $\sroot$, for short.  
Since $\reflection(\newvar{i})=\newvar{i}-\sroot$ for $i\leq4$ and
$\reflection(\newvar{i})=\newvar{i}+\sroot$ for $i\geq4$, 
we can compute $\reflection(a_i)$ and $\reflection(b_i)$ easily, and
this yields the following:
\begin{center}
\renewcommand{\arraystretch}{1.2}
\begin{tabular}{@{}p{30em}@{}}
$\reflection(c_2) = 
c_2
,$
\\
$\reflection(c_3) = 
c_3
+2(a_2-b_2)\sroot
,$
\\
$\reflection(c_4) \equiv 
c_4
+3(a_3-b_3)\sroot
-3(a_2+b_2)\sroot^2 \mod(\sroot^4)
,$
\\
$\reflection(c_5) \equiv 
c_5
+4(a_4-b_4)\sroot
-2(a_3+b_3)\sroot^2 \mod(\sroot^4)
,$
\\
$\reflection(c_6) \equiv 
c_6
+(a_3b_2-a_2b_3)\sroot
-2a_2b_2\sroot^2
-2(a_3-b_3)\sroot^3 \mod(\sroot^4)
,$
\\
$\reflection(c_8) \equiv 
c_8
+(a_4 b_3-a_3 b_4)\sroot
+(a_4b_2+a_2b_4-a_3b_3)\sroot^2
-(a_3 b_2-a_2 b_3)\sroot^3 \mod(\sroot^4)
.$
\\
\end{tabular}

\end{center}
We then conclude a generator $\generator{i}$ in degree $i$ satisfies 
the following by computing modulo $(\sroot^2)$:  

\begin{center}
\renewcommand{\arraystretch}{1.2}
\begin{tabular}{@{}p{30em}@{}}
$\generator{4}=c_2$, \\
$\generator{12}\equiv c_6-\frac{1}{6}c_2c_4+\frac{1}{8}c_3^2\mod(a_1),$
\\
$\generator{16}\equiv c_8-\frac{1}{4}c_2c_6-\frac{1}{8}c_3c_5+\frac{1}{12}c_4^2\mod(a_1).$
\\
\end{tabular}
\end{center}


\end{document}